\documentclass[11pt,fleqn]{amsart}
\usepackage[paper=a4paper]
  {geometry}

\usepackage[utf8]{inputenc}
\usepackage[english]{babel}

\usepackage{mathabx}
\usepackage{amssymb}
\usepackage[alphabetic,initials]{amsrefs}
\usepackage[osf,noBBpl]{mathpazo}
\usepackage[poly,arrow,curve,matrix]{xy}
\usepackage[shortlabels]{enumitem}
\usepackage{stmaryrd}

\setlist[enumerate,1]{label = \emph{(\roman*)},
  ref = \emph{\theenumii.\roman*}}
\usepackage{tabularx} 
\usepackage{tikz}
\usepackage{mathdots}
\usetikzlibrary{calc,arrows.meta,backgrounds,decorations.pathmorphing,cd,positioning}
\usepackage{ytableau}

\newcommand\CC{\mathbb C}

\newcommand\ZZ{\mathbb Z}
\newcommand\II{\mathbb I}

\newcommand\JJ{\mathbb J}
\newcommand\LL{\mathbb L}
\renewcommand\ll{{\boldsymbol \ell}}

\renewcommand\to{\longrightarrow}
\renewcommand\phi{\varphi}
\renewcommand\epsilon{\varepsilon}

\newtheorem{thm}{Theorem}[section]
\newtheorem{lem}[thm]{Lemma}
\newtheorem{prop}[thm]{Proposition}
\newtheorem{cor}[thm]{Corollary}

\newtheorem*{thm*}{Theorem}
\newtheorem*{cor*}{Corollary}
\theoremstyle{definition}
\newtheorem{defn}[thm]{Definition}
\newtheorem*{defn*}{Definition}
\theoremstyle{remark}

\theoremstyle{remark}

\newcommand\VV{\mathbb V}

\newcommand\gl{\mathfrak{gl}}
\renewcommand\sl{\mathfrak{sl}}

\renewcommand\SS{\mathbf S}

\newcommand\Id{\mathsf{Id}}

\newcommand\interval[1]{\llbracket{#1}\rrbracket}
\newcommand\lie[1]{\mathfrak{#1}}
\newcommand\op{\mathsf{op}}
\newcommand\wt[1]{\boldsymbol{#1}}

\DeclareMathOperator\soc{soc}
\DeclareMathOperator\rad{rad}
\DeclareMathOperator\supp{supp}
\newcommand\vspan[1]{\operatorname{span}{\{#1\}}}
\newcommand\Pieri[2]{\operatorname{Pieri}{(#1},{#2)}}
\newcommand\map[4]{\phi(#3,#4,#1,#2)}

\title{The Pieri Rule at Infinity}

\date{}

\author{Ivan Penkov}
\address{Ivan Penkov. Constructor University, 28759 Bremen, Germany}
\email{ipenkov@constructor.university}

\author{Pablo Zadunaisky}
\address{Pablo Zadunaisky. Constructor University, 28759 Bremen, Germany and IMAS -- CONICET, Buenos Aires, Argentina.}
\email{pzadunaisky@conicet.gov.ar}

\begin{document}

\begin{abstract}
  We study the structure of tensor products of $\gl(\infty) = \varinjlim 
  \gl(n)$-modules $\mathbf L(\wt\lambda) \otimes \mathbf F$ where 
  $\mathbf L(\wt\lambda)$ is a simple integrable highest weight module and 
  $\mathbf F$ is a simple integrable weight multiplicity-free module. Both 
  $\mathbf L(\wt\lambda)$ and $\mathbf F$ are infinite dimensional, in 
  particular 
  $\mathbf F$ can be a Fock module. Similar tensor products of $\gl(n)$-modules 
  are semisimple and their simple constituents are described by the classical 
  Pieri rule. We prove that a $\gl(\infty)$-module $\mathbf M:= \mathbf 
  L(\wt\lambda) \otimes \mathbf F$ is semisimple only in relatively trivial 
  cases, and is indecomposable otherwise. Our main results are a description 
  of the simple constituents of $\mathbf M$, and the construction of a linkage 
  filtration on $\mathbf M$ that provides information on when two simple 
  constituents of $\mathbf M$ are linked. Using the linkage filtration, we 
  compute the socle and radical filtrations of $\mathbf M$, and determine when 
  $\mathbf M$ is rigid.

\end{abstract}

\maketitle

\section{Introduction}

The Pieri rule is a cornerstone of the representation theory of the Lie algebra 
$\gl(n)$. It is fudamental for the study of the category of finite-dimensional 
$\gl(n)$-modules as a tensor category. In the simplest case of a tensor product 
$L(\wt\lambda) \otimes V$, where $L(\wt\lambda)$ is a simple $\gl(n)$-module 
with integral highest weight $\wt\lambda$ and $V$ is the defining 
$\gl(n)$-module, the Pieri rule states that $L(\wt\lambda) \otimes V$ 
decomposes as $\bigoplus_{i} L(\wt\lambda + \epsilon_i)$ where $\epsilon_j$ 
are the weights of $V$ and the sum runs over all $i$ such that $\wt\lambda + 
\epsilon_i$ is a dominant weight.

Our aim in this paper is to understand an analogue of Pieri's rule for the 
infinite-dimensional Lie algebra $\gl(\infty) := \varinjlim \gl(n)$. We consider
a general tensor product $\mathbf L(\wt\lambda) \otimes \mathbf F$, where 
$\mathbf L(\wt\lambda)$ is any simple integrable highest weight 
$\gl(\infty)$-module with highest weight $\wt\lambda$  (with respect to
some splitting Borel subalgebra of $\gl(\infty)$) and $\mathbf F$ is any 
simple weight multiplicity-free $\gl(\infty)$-module admiting a highest weight. 
Both $\mathbf F$ and $\mathbf L(\wt\lambda)$ are infinite dimensional.

Since the category of integrable finite-dimensional $\gl(n)$-modules is 
semisimple according to a celebrated theorem of H. Weyl from 1925, the 
classical Pieri rule is usually considered as a combinatorial rule in the 
language of Young diagrams. However, in the infinite-dimensional setting we 
consider, a priori there is no reason for the tensor products $\mathbf 
L(\wt\lambda) \otimes \mathbf F$ to be semisimple, even when $\mathbf F$ is the 
defining representation $\VV$ of $\gl(\infty)$. Moreover, we show that in most 
interesting cases modules of the form $\mathbf L(\wt\lambda) \otimes \mathbf 
F$, or even $\mathbf L(\wt\lambda) \otimes \VV$, are actually indecomposable. 
In the particular case when $\mathbf L(\wt\lambda)$ is a tensor module (see 
\cites{PS11b}) and $\mathbf F = \SS^d\VV$, this has been proved earlier by A. 
Chirvasitu as part of an unpublished study undertaken jointly with the 
first author. 

As already mentioned, in the current paper we take a much broader approach by 
letting $\mathbf L(\wt\lambda)$ be any simple integrable highest weight 
$\gl(\infty)$-module and $\mathbf F$ be any simple highest weight 
multiplicity-free $\gl(\infty)$-module. The classification of simple weight 
multiplicity-free $\sl(\infty)$-modules in \cite{GP20} implies that, up to 
tensor multiplication with a one-dimensional module, a simple highest weight 
multiplicity-free $\gl(\infty)$-module is a symmetric or exterior power of 
one of the defining representations $\VV$ and $\VV_*$, or is a Fock module. The 
latter are spaces of semi-infinite wedge vectors endowed with natural 
$\gl(\infty)$-actions. 

To study the modules $\mathbf L(\wt\lambda) \otimes \mathbf F$ we proceed as 
follows. Suppose we are given a sequence of finite-dimensional Lie subalgebras 
$\lie g_n \subset \gl(\infty)$ such that $\gl(\infty) = \bigcup_{n \geq 1} \lie 
g_n$. Then there exist simple finite-dimensional $\lie g_n$-modules $L_n$ and 
$F_n$, with $\lie g_n$-linear maps $L_n \to L_{n+1}$ and $F_n \to F_{n+1}$
such that $\varinjlim L_n \cong \mathbf L(\wt\lambda)$ and $\varinjlim F_n = 
\mathbf F$, so $\mathbf L(\wt\lambda) \otimes \mathbf F = \varinjlim L_n 
\otimes F_n$. The maps $L_n \to L_{n+1}$ and $F_n \to F_{n+1}$ are easy to 
describe since the source is a simple $\lie g_n$-module and the map is unique
up to a constant. Understanding the tensor product of the maps is harder. Our 
strategy is to single out one particular exhaustion of $\mathbf L(\wt\lambda)
\otimes \mathbf F$ making a complete understanding of the tensor product 
maps possible. 

Our main results are a characterization of the simple constituents of $\mathbf
L(\wt\lambda) \otimes \mathbf F$ and the construction of a natural filtration 
which we call the linkage filtration. This filtration shows when two simple 
constituents are linked and enables us to compute the socle and 
radical filtration of $\mathbf L(\wt\lambda) \otimes \mathbf F$. In particular 
we give a necessary and sufficient condition for the socle and radical 
filtrations to be finite, as well for both filtrations to coincide, i.e., for 
the module to be rigid. We provide also 
several examples showing that the type and length of the linkage filtration can 
vary significantly.

The paper is structured as follows. In section $2$ we recall some necessary 
general background on the representation theory of $\gl(\infty)$. In section
$3$ we introduce some basic definitions, state our main results without proof, 
and give various examples of how they apply. Section $4$ focuses on the 
transition maps $L_n \otimes F_n \to L_{n+1} \otimes F_{n+1}$ and shows that 
under certain hypotheses we can describe them completely. In section $5$ we 
introduce exhaustions of $\gl(\infty)$ which satisfy one of the hypotheses of 
the previous section at each step. Finally, the proofs of the main results are 
given in section $6$.

As an outlook we would like to mention that the results of this paper should be 
instrumental in further studies of various tensor categories of integrable 
$\gl(\infty)$-modules. This applies in particular to the integrable part of the
category $\mathcal O_{LA}^{\lie l}$ introduced in \cite{Zadunaisky22}, in which 
certain indecomposable injective modules are of Pieri type.

\textbf{Acknowledgement:} The work of both authors was supported in part by  
DFG grant PE 980/9-1. PZ is a CONICET researcher.

\section{Notation and Background}

\subsection{Conventions on ordered sets}
Let $(P, \preceq)$ be a partially ordered set (poset). Given $p,q\in P$ we 
denote by $\interval{p,q}$ the set $\{r \in P \mid p \preceq r \preceq q\}$. 
The set $P$ is \emph{locally finite} if for every $p,q \in P$ the interval 
$\interval{p,q}$ is a finite set. We say that $P$ is $\emph{ranked}$ if there 
is a function $f: P \to \ZZ$ such that $p \prec q$ implies $f(p) < f(q)$; $P$ 
is \emph{locally ranked} if every interval $\interval{p,q}$ is a ranked poset 
with 
the induced order. A \emph{chain} in $P$ is a subset which is totally ordered 
with the induced order, and the \emph{rank} of $P$ is the supremum of the 
cardinalities of all chains in $P$. The poset $P$ is a \emph{lattice} if any 
two elements have a least upper bound, called their \emph{join}, and a maximal 
lower bound, called their \emph{meet}. In the sequel we write $\interval n = 
\interval{1,n}$ for any $n \in \ZZ_{>0}$. 

The partial order $\preceq$ on $P$ is \emph{total} or \emph{linear} if 
every pair of elements is comparable. Let $P$ be a linearly ordered set. An 
\emph{initial subset} is a subset $I \subset P$ such that $a \in I$ and $a' 
\prec a$ implies $a' \in I$. A \emph{terminal subset} is a subset $F \subset P$ 
with the opposite property: if $a \in F$ and $a' \in P$ satisfies $a' \succ a$ 
then $a' \in F$. If $\preceq$ is a linear order on $P$ we denote by $P^\op$ the 
set $P$ with the opposite order. If $P$ is a partially ordered set isomorphic
to a subset of $\ZZ$, and $p \in P$, we write $p+1$ and $p-1$ for the immediate 
successor and the immediate predecessor of $p$, assuming these exist.

Given linearly ordered sets $P$ and $Q$, we denote by $P 
\sqcup Q$ the set $P \times \{0\} \cup Q \times \{1\}$ with the total 
lexicographic order. More generally, given a linearly ordered set $I$ and a 
family of linearly ordered sets $(P_i)_{i \in I}$, we denote by $\bigsqcup_{i 
\in I} P_i$ the set $\bigcup_{i \in I} P_i \times \{i\}$ with the total 
lexicographic order. We will sometimes identify the set $P_i$ with $P_i \times 
\{i\}$ when the context makes this clear, in particular when $P$ and $Q$ are 
disjoint.

\subsection{The Lie algebra $\gl(\infty)$}
The base field is the field of complex numbers $\CC$. We fix a 
countable-dimensional space $\VV$ and a countable-dimensional subspace $\VV_*$ 
of $\VV^* := \hom_\CC(\VV,\CC)$ such that the pairing $\VV \otimes \VV_* \to 
\CC$ (induced by the pairing $\VV \otimes \VV^* \to \CC$) is non-degenerate. 
The space $\VV \otimes \VV_*$ is an associative $\CC$-algebra: the 
multiplication is given by
\[
  \left(\sum_i v_i \otimes \phi_i\right)
    \left(\sum_j v'_j \otimes \phi'_j\right)
    = \sum_{i,j} \phi_i(v'_j)v_i \otimes \phi_j
\]
for $v_i, v'_j \in \VV$ and $\phi_i, \phi'_j \in \VV_*$. By definition, 
$\gl(\infty)$ is the Lie algebra determined by the associative algebra $\VV 
\otimes \VV_*$, and $\VV$ and $\VV_*$ are the \emph{defining representations} 
of $\gl(\infty)$. Denoting the pairing $\VV \otimes 
\VV_* \to \CC$ by $\operatorname{tr}$, we define the Lie algebra $\sl(\infty)$ as 
$\ker \operatorname{tr}$. In this way we obtain a simple Lie subalgebra of 
codimension $1$ in $\gl(\infty) = \VV \otimes \VV_*$. Note however that the exact sequence of $\sl(\infty)$-modules
\[
  0 \to \sl(\infty) \to \VV \otimes \VV_* \stackrel{\operatorname{tr}}{\to} \CC \to 0
\]
does not split.

According to a theorem of G. Mackey \cite{Mackey45}, there exists a basis 
$\{e_k\}_{k \geq 1}$ of $\VV$ such that its dual system $\{e_k^*\}_{k \geq 1}$
is a basis of $\VV_*$. Setting $V_n := \vspan{e_k \mid 1 \leq k \leq n}$ and 
$V_n^* := \vspan{ e_k^* \mid 1 \leq k \leq n}$, we obtain $\gl(\infty)$ as the
inductive limit of the finite dimensional subalgebras $\gl(V_n) \subset 
\gl(\infty)$. A \emph{splitting Borel subalgebra} $\lie b$ of $\gl(\infty)$ is 
the inductive limit of a chain of Borel subalgebras $\lie b_n \subset \gl(V_n)$ 
for any choice of $V_n$ as above. It is well known that any such $\lie b$ 
contains the limit $\lie h = \varinjlim \lie h_n$ of some Cartan subalgebras 
$\lie h_n 
\subset \lie b_n$. We fix the pair $\lie h \subset \gl(\infty)$, and will 
consider various splitting Borel subalgebras $\lie b$ containing $\lie h$. In 
particular we will assume that $V_n$ are chosen so that $\lie h_n = \lie h 
\cap \gl(V_n)$ is a Cartan subalgebra for all $n$.

A \emph{weight $\gl(\infty)$-module} $\mathbf N$ is an $\gl(\infty)$-module which 
is semisimple over $\lie h$:
\[\mathbf N = \bigoplus_{\wt\mu \in \lie h^*} \mathbf N_{\wt\mu}\]
where $\lie h^* = \hom(\lie h, \CC)$ and 
\[\mathbf N_{\wt\mu} := \{v \in \mathbf N \mid hv = \wt\mu(h)v \  \forall 
  h \in \lie h\}.\]
The elements of $\lie h^*$ are the \emph{weights} of $\gl(\infty)$, and the 
\emph{support} $\supp \mathbf N$ of the module $\mathbf N$ is defined by setting 
$\supp \mathbf N := \{\wt\mu \in \lie h^* \mid  \mathbf N_{\wt\mu} \neq 0\}$. The 
nonzero spaces $\mathbf N_{\wt\mu}$ are the \emph{weight spaces} of the module 
$\mathbf N$. The defining representation $\VV$
is a weight $\gl(\infty)$-module and its weights are $\{\epsilon_i\}_{i \in 
\II}$, where $\II$ is a countable set which will be fixed throughout this paper.
As usual we denote a weight $\wt\lambda \in \lie h^*$ by $\sum_{i \in \II} 
\wt\lambda_i \epsilon_i$, where $\wt\lambda_i = \wt\lambda(E_{i,i})$; notice that
$\wt\lambda$ is well defined even when the sum is infinite. We also fix the 
following notation: we let $\lie h^\circ$ be the $\CC$-span of the set
$\{\epsilon_i\mid i \in \II\}$, and given $A \subset \II$ put $\epsilon_A := 
\sum_{i \in A} \epsilon_i$.

Suitably chosen outer automorphisms of $\gl(V_n)$ extend to an 
outer automorphism $\tau$ of $\gl(\infty)$ acting by $-\Id$ on $\lie h$, and 
any $\gl(\infty)$-module $\mathbf N$ can be twisted by this automorphism. We 
denote the result by $\mathbf N_*$. If $\mathbf N$ is a weight module then 
$\supp \mathbf N_* = - \supp \mathbf N$. In particular, $\VV_*$ is nothing but
$\VV$ twisted by the automorphism $\tau$, and $\supp \VV_* = \{-\epsilon_i
\}_{i \in \II}$.

The adjoint representation of $\gl(\infty)$ is a weight module, and its weights 
are $\{\epsilon_i - \epsilon_j\}_{i,j \in \II}$. The nonzero weights of the
adjoint representation are the \emph{roots} of $\gl(\infty)$, and the zero 
weight space of $\gl(\infty)$ is $\lie h$. A splitting Borel subalgebra $\lie b 
\supset \lie h$ is determined by a linear order $\preceq$ on the set $\II$:
\[\lie b_{\preceq} := \lie h \oplus \bigoplus_{i \prec j} 
  \gl(\infty)_{\epsilon_i - \epsilon_j}.\]
The \emph{positive roots} of $\lie b$ are the roots $\epsilon_i - \epsilon_j$ 
with $i \prec j$, and the \emph{negative roots} are $\epsilon_i - \epsilon_j$
with $i \succ j$; the \emph{simple roots} of $\lie b$ are the positive roots 
which can not be written as a sum of two positive roots. In contrast with the 
finite-dimensional case, not every positive root is in the $\ZZ_{>0}$-span of 
the simple roots of $\lie b$. The linear order $\preceq$ which defines 
$\lie b$ determines a partial order on $\lie h^*$, again denoted $\preceq$, 
where $\wt\lambda \preceq \wt\mu$ if and only if $\wt\mu - \wt\lambda$ is a sum 
of positive roots. Given a splitting Borel $\lie b$, we denote by $\lie b^{\op}$
the unique splitting Borel subalgebra whose positive roots are the negative 
roots of $\lie b$.

\subsection{Integrable highest weight $\gl(\infty)$-modules}
A weight $\wt\lambda \in \lie h^*$ is \emph{integral} if $\wt\lambda_i \in \ZZ$
for all $i \in \II$. A $\gl(\infty)$-module $\mathbf N$ is \emph{integrable} if 
its restriction to $\gl(V_n)$ is a direct sum of finite-dimensional 
$\gl(V_n)$-modules with integral weights. This means that $\mathbf N$ 
integrates to a $\mathsf{GL}(V_n)$-module for any $n$. 

A weight module $\mathbf N$ is a \emph{$\lie b$-highest weight module} for a 
given Borel 
subalgebra $\lie b$ if $\mathbf N$ is generated by a single weight vector $v$ 
such that $[\lie b, \lie b]v = 0$. The weight of the vector $v$ is the 
\emph{highest weight} of $\mathbf N$. The usual construction shows that given a 
linear order $\preceq$ on $\II$ and a weight $\wt\lambda \in \lie h^*$ then 
there exists a $\lie b_{\preceq}$-highest weight module with highest weight 
$\wt\lambda$, which we denote by $\mathbf L_{\lie b_{\preceq}}(\wt\lambda)$. 

We say that $\wt\lambda \in \lie h^*$ is \emph{$\lie b_{\preceq}$-dominant} if 
$\wt\lambda_i - \wt\lambda_{j} \in \ZZ_{\geq 0}$ whenever $i 
\prec j$. The simple integrable modules which are $\lie b_{\preceq}$-highest 
weight modules are precisely the modules $\mathbf L_{\lie b_{\preceq}}
(\wt\lambda)$ such that $\wt\lambda$ is integral and $\lie 
b_{\preceq}$-dominant. In the rest of the paper we assume automatically that 
all weights are integral. We only consider integrable simple highest weight 
modules, and simplify the notation from $\mathbf 
L_{\lie b_{\preceq}}(\wt\lambda)$ to $\mathbf L(\wt\lambda)$ by assuming that 
$\preceq$ is any linear order on $\II$ for which $\wt\lambda$ is $\lie 
b_{\preceq}$-dominant. In addition, it is easy to check that there is an 
isomorphism $\mathbf L(\wt\lambda) \cong \mathbf L(\wt\mu)$ if and only if 
there is a bijective map $\sigma: \II \to \II$ which is the identity outside of 
a finite set and such that $\wt\lambda_i = \wt\mu_{\sigma(i)}$.

\subsection{Tensor modules}
A $\gl(\infty)$-module is a \emph{tensor module} if it is isomorphic to
a submodule of a direct sum of finitely copies of the tensor algebra $T(\VV 
\oplus \VV_*)$. For instance, 
$\sl(\infty)$ is a simple tensor module, being a submodule of $\VV \otimes 
\VV_*$. Moreover, all simple tensor modules are highest weight modules with 
respect to the Borel subalgebra corresponding to the \emph{programmers order} 
on $\II$, defined as $\ZZ_{>0} \sqcup \ZZ_{<0}$ where the orders of 
$\ZZ_{>0}$ and $\ZZ_{<0}$ are the usual ones and $i \prec j$ whenever $i \in 
\ZZ_{>0}, j \in \ZZ_{<0}$. The highest weight of a simple tensor module 
is of the form $\wt\lambda = \lambda_1 \epsilon_1 + \cdots + 
\lambda_n \epsilon_n - \mu_{k} \epsilon_{-k} - \cdots -\mu_{1}\epsilon_{-1}$ 
for some $n,k \in \ZZ_{\geq 0}$ and $\lambda_1 \geq \cdots \geq \lambda_n, 
\mu_1 \geq \cdots \geq \mu_k$. In particular the highest weight of $\sl(\infty)$
is $\epsilon_1 - \epsilon_{-1}$ \cites{PS11b,HP22}. More generally, it follows 
from the results in the previous section that the tensor module $\mathbf 
L(\wt\lambda)$ is a highest weight module for a Borel subalgebra $\lie 
b_{\preceq}$ if and only if the order $\preceq$ has an initial subset 
isomorphic to $\interval n$ and a terminal subset isomorphic to $\interval k$.

\subsection{Weight multiplicity-free $\gl(\infty)$-modules}
A weight $\gl(\infty)$-module $\mathbf N$ is \emph{weight multiplicity-free} if
$\dim \mathbf N_{\wt\mu} = 1$ for all $\wt\mu \in \supp \mathbf N$. In 
\cite{GP20} simple weight multiplicity free $\sl(\infty)$-modules with a 
highest weight are classified. Up to isomorphism, these are the symmetric and 
exterior powers $\SS^d\VV, \bigwedge^d \VV, \SS^d\VV_*, \bigwedge^d \VV_*$, 
and the Fock modules $\bigwedge^A \VV$ which we now define. Any infinite set 
$A \subset \II$ with infinite complement $B = \II \setminus A$ determines a 
family of Borel subalgebras $\lie b_{\preceq}$ for which $A$ is an initial 
subset of $(\II, \preceq)$. For all such $\lie b_{\preceq}$, the simple $\lie 
b_{\preceq}$-highest weight modules $\mathbf L_{\lie b_{\preceq}}(\epsilon_A)$
are canonically isomorphic, and we define the module $\bigwedge^A \VV$ as any 
of those modules.
A weight in the support of $\bigwedge^A \VV$ is of the form $\sum_{i \in B} 
\epsilon_i$, where $B \subset \II$ and $\# A \setminus B = \# B \setminus A < 
\infty$. The module $\bigwedge^A \VV$ is minuscule in the sense that every 
weight is a highest weight for a suitable $\lie b_{\preceq}$.

Since we work with $\gl(\infty)$ and not $\sl(\infty)$ we need to adapt the 
classification from \cite{GP20} to $\gl(\infty)$. It is easy to check that a 
simple highest weight $\gl(\infty)$-module which is multiplicity-free is 
isomorphic to one of the modules $\SS^d\VV, \bigwedge^d \VV, \SS^d\VV_*, 
\bigwedge^d \VV_*$ or $\bigwedge^A \VV$, possibly tensored by $\mathbf 
D^{\otimes n}$ for any $n \in \ZZ$. Here $\mathbf D$ is the one-dimensional 
$\gl(\infty)$-module of weight $\sum_i \epsilon_i$. We use the convention that 
$\mathbf D^{\otimes - n} = \mathbf D_*^{\otimes n}$. Note that $\mathbf D$ is 
not a tensor module.

\subsection{Compatible orders}
The \emph{support} $\supp \wt\lambda$ of a weight $\wt\lambda$ is the set of all 
$i \in \II$ such that $\wt\lambda_i \neq 0$. 
We will often identify $\wt\lambda$ with the function $i \in \II \mapsto 
\wt\lambda_i \in \ZZ$, and so use notation like $\wt\lambda^{-1}(a) = \{i \in \II 
\mid \wt\lambda_i = a\}$ for $a \in \ZZ$ and $\wt\lambda(Z) = \{\wt\lambda_i \mid 
i \in Z\}$ for $Z \subset \II$. Note that $\wt\lambda$ is $\lie 
b_{\preceq}$-dominant if and only if $\II = \bigsqcup_{a \in \ZZ^\op} 
\wt\lambda^{-1}(a)$, where the order on $\wt\lambda^{-1}(a)$ is induced from 
$\II$. 

\begin{defn}
\label{defn:compatible}
Let $\wt\lambda \in \lie h^*$ and let $\mathbf F$ be one of $\SS^d\VV, 
\bigwedge^d \VV$ or $\bigwedge^A \VV$. We say that a total order $\preceq$ on 
$\II$ is \emph{($\wt\lambda, \mathbf F$)-compatible} if it satisfies the 
following 
conditions:
\begin{enumerate}
  \item $\wt \lambda$ is $\lie b_{\preceq}$-dominant.
  \item For every $a \in \ZZ$ the poset $\wt\lambda^{-1}(a)$ is order-isomorphic 
    to a subset of $\ZZ$.
  \item For each $\wt\mu \in \supp F$ the set $\wt\lambda^{-1}(a) \cap \supp 
    \wt\mu$ has a maximum and the set $\wt\lambda^{-1}(a) \setminus \supp \wt\mu$ 
    has a minimum.
\end{enumerate}
\end{defn}

\begin{lem}
For $\wt\lambda \in \lie h^*$ and $\mathbf F$ be as above, there exists a 
$(\wt\lambda,\mathbf F)$-compatible total order $\preceq$. Furthermore, if 
$\mathbf L(\wt\lambda)$ and $\mathbf F$ are both $\lie b_{\preceq'}$-highest 
weight modules for some total order $\preceq'$ then $\mathbf F$ is also a 
$\lie b_{\preceq}$-highest weight module.
\end{lem}
\begin{proof}
We first show the existence of a $(\wt\lambda,\mathbf F)$-compatible order 
$\preceq$
on $\II$, i.e., a linear order $\preceq$ such that $\II = \bigsqcup_{a \in 
\ZZ^\op} \wt\lambda^{-1}(a)$. For this it is enough to assign a total order on 
each $\wt\lambda^{-1}(a)$ so that conditions $(ii)$ and $(iii)$ of Definition
\ref{defn:compatible} are satisfied.

For $\mathbf F = \SS^d\VV, \bigwedge^d \VV$ we impose on $\wt\lambda^{-1}(a)$ 
an order making it isomorphic to a subset of $\ZZ_{>0}$; then 
$\wt\lambda^{-1}(a) \cap \supp \mu$ is a finite set and hence always has a 
maximum, while its complement has a minimum by the well-ordering principle. 
For $\mathbf F = \bigwedge^A \VV$ we impose on $\wt\lambda^{-1}(a) 
\cap A$ an order-isomorphic to that of a subset of $\ZZ_{<0}$, on 
$\wt\lambda^{-1}(a) \setminus A$ an order-isomorphic to that of a subset of  
$\ZZ_{>0}$, and put $\wt\lambda^{-1}(a) = (\wt\lambda^{-1}(a) \cap A) \sqcup 
(\wt\lambda^{-1}(a) \setminus A)$. 
If $\wt\mu \in \supp F$ then $\supp \wt\mu$ differs from $A$ in finitely many 
elements, and since $\wt\lambda^{-1}(a) \cap A$ has a minimum and 
$\wt\lambda^{-1}(a) \setminus A$ has a maximum, the sets $\wt\lambda^{-1}(a) 
\cap \supp \wt\mu$ and $\wt\lambda^{-1}(a) \setminus \supp \wt\mu$ also have 
respectively a minimum and a maximum. This completes the proof of the existence 
of a $(\wt\lambda, \mathbf F)$-compatible order.

Suppose now $\mathbf L(\wt\lambda)$ and $\mathbf F$ are both $\lie 
b_{\preceq'}$-highest weight modules. If $\mathbf F = \SS^d\VV$ then $(\II,
\preceq')$ must have a minimum, which implies that $\wt\lambda(\II) \subset 
\ZZ$ must have a maximum $m \in \ZZ$. Then, if $i$ is the minimal element of 
$\wt\lambda^{-1}(m)$, any nonzero vector in $\SS^d\VV$ of weight $d \epsilon_i$ 
is a 
$\lie b_{\preceq}$-highest weight vector. If $\mathbf F = \bigwedge^d \VV$ then 
$\preceq'$ has an initial sequence of $d$ elements $i'_1 \prec' i'_2 \prec' 
\cdots \prec' i'_d$. Since $\wt\lambda$ is $\lie b_{\preceq'}$-dominant we  
have $\wt\lambda_{i'_1} \geq \wt\lambda_{i'_2} \geq \cdots \geq 
\wt\lambda_{i'_d}$, and the order $\preceq$ constructed above also has an 
initial sequence of $d$ elements $i_1 \prec i_2 \prec \cdots \prec i_d$. Now it 
is clear that any nonzero vector of weight $\epsilon_{i_1} + \epsilon_{i_2} + 
\cdots + \epsilon_{i_d}$ is a $\lie b_{\preceq}$-highest weight vector of 
$\bigwedge^d \VV$. 

Finally, let $\mathbf F = \bigwedge^A \VV$. This is a $\lie 
b_{\preceq'}$-highest weight module if and only if $A$ differs from an initial
subset of $(\II,\preceq')$ by finitely many elements; without loss of 
generality we can and will assume that $A$ is an initial subset of $(\II,
\preceq')$. It follows that $A = (A \cap \wt\lambda^{-1}(r)) \cup \bigcup_{a < 
r} \wt\lambda^{-1}(a)$ for some $r \in \ZZ$. By construction this is an 
initial subset in $(\II,\preceq)$, so $\mathbf F$ is a $\lie 
b_{\preceq}$-highest weight module.
\end{proof}

\subsection{Filtrations}
We adopt the following terminology concerning filtrations. A \emph{filtration} 
of a module $\mathbf N$ is a chain of distinct submodules $\{F_\ll\mathbf 
N\}_{\ll \in \LL}$ where $\LL$ is a linearly ordered set isomorphic to a subset 
of $\ZZ$. The quotients $\overline F_{\ll+1} \mathbf N := F_{\ll+1} \mathbf N / 
F_\ll \mathbf N$ are the \emph{layers} of the filtration. A filtration is 
\emph{exhaustive} if $\bigcup_{\ll \in \LL} F_\ll \mathbf N = \mathbf N$ and $
\emph{separated}$ if $\bigcap_{\ll \in \LL} F_\ll \mathbf N = \{0\}$. If $\mathbf 
L$ is a simple module its \emph{multiplicity in $\mathbf N$}, denoted by 
$[\mathbf N: \mathbf L]$, is the supremum of the number of times $\mathbf L$ 
appears as a layer in a finite filtration of $\mathbf N$. By definition 
$[\mathbf N: \mathbf L] \in \ZZ_{\geq 0} \cup \{\infty\}$. 

The module $\mathbf N$ is said to be \emph{multiplicity free} if $[\mathbf 
M: \mathbf L] \leq 1$ for any simple module $\mathbf L$.
A \emph{composition series} is an exhaustive and discrete filtration in which 
every layer is simple, and such that for every simple subquotient $\mathbf L$ of 
$\mathbf N$, there are $[\mathbf N: \mathbf L]$ layers isomorphic to $\mathbf L$. 
If $\mathbf N$ is multiplicity free we say that a simple constituent 
$\mathbf L$ is \emph{linked} to a simple constituent $\mathbf L'$ if for every
submodule $\mathbf S \subset \mathbf N$ such that $[\mathbf S: \mathbf L] = 1$
then $[\mathbf S: \mathbf L'] = 1$.

Recall that the \emph{socle} of $\mathbf N$, denoted $\soc \mathbf N$, is the 
maximal semisimple submodule of $\mathbf N$. The \emph{socle filtration} of 
$\mathbf N$ is defined via induction by setting $\soc^{(0)} \mathbf N := \{0\}$ 
and $\soc^{(i+1)}\mathbf N := p^{-1} (\soc \left(\mathbf N /  
\soc^{(i)} \mathbf N\right))$, where $p: \mathbf N \to \mathbf N/ \soc^{(i)} 
\mathbf N$ is the canonical projection. 

The \emph{radical} of $\mathbf N$, denoted $\rad \mathbf N$, is the intersection 
of the kernels of all maps $\pi: \mathbf N \to \mathbf L$ with $\mathbf L$ a 
simple $\gl(\infty)$-module. The \emph{radical filtration} of $\mathbf N$ is 
similarly defined by induction, setting $\rad^{(0)} \mathbf N = 
\mathbf N$ and $\rad^{(i+1)} \mathbf N = \rad \left( \rad^{(i)} 
\mathbf N\right)$.  
A module $\mathbf N$ is \emph{rigid} if its socle filtration is exhaustive, its 
radical filtration is separated, and they coincide.

We say that a module $\mathbf N$ has a well-defined \emph{Loewy length} if it has 
at least one exhaustive and discrete filtration whose layers are semisimple. The 
Loewy length of the module is then the minimum of the lengths of these 
filtrations, and is denoted by $\mathcal L \ell(\mathbf N)$. By a classical 
result the socle filtration of the module $\mathbf N$ is finite and exhaustive 
if and only if its radical filtration is finite and discrete, and in this case 
both filtrations have the same length, which is the Loewy length of the module. 
A proof of this fact for finite length modules can be found in \cite{ASSI}*{
Chapter V, Proposition 1.3}, and the proof adapts easily to our more general 
case.


\section{Main statements}
Fix a weight $\wt\lambda \in \lie h^*$.
From this point on, we assume that $\mathbf F$ is either $\bigwedge^d \VV, 
\SS^d\VV$ or $\bigwedge^{A} \VV$, and that the fixed  total order $\preceq$ 
on $\II$ is $(\wt\lambda, \mathbf F)$-compatible. We set $\lie b = \lie 
b_{\preceq}$ and $\mathbf M := \mathbf L(\wt\lambda) \otimes \mathbf F$.

\subsection{The sets $\JJ$ and $\JJ_{\infty}$}
We introduce an equivalence relation $\sim$ in $\II$: we write $i \sim i'$ if 
and only if $\wt\lambda_i = \wt\lambda_{i'}$ and set $\JJ := \II/\sim$. By 
definition every equivalence class $j \in \JJ$ is a set $\wt\lambda^{-1}(a)$
for some $a \in \ZZ$, and thus an interval in the order $\preceq$. This implies 
that $\JJ$ inherits a total order from $\II$, and that the map $i \in \II 
\mapsto \wt\lambda_i \in \ZZ^\op$ induces an order-preserving injection of 
$\JJ$ into $\ZZ^\op$. 

Denote by $\pi_{\JJ}: \II \to \JJ$ the natural projection, and let $s: \JJ \to 
\II$ be any section. We define a further equivalence relation $\sim_\infty$ in 
$\JJ$: if $j \preceq j'$ then $j \sim_\infty j'$ if and only if the interval 
$\interval{s(j),s(j')} \subset \II$ is finite. The definition of a $(\wt\lambda,
\mathbf F)$-compatible order guarantees that if $i \prec i'$ and $\pi_\JJ(i) = 
\pi_{\JJ}(i')$ then $\interval{i,i'}$ is finite, hence the equivalence relation
$\sim_\infty$ is independent of the choice of $s$. Similarly to $\JJ$, the set 
$\JJ_\infty := \JJ/\sim_\infty$ inherits a total order and is order-isomorphic 
to a subset of $\ZZ$. 


Denote by $\CC^{(\JJ)}$ the set of all finitely supported functions $\JJ \to 
\CC$. We define a linear map from $\lie h^\circ$ to 
$\CC^{(\JJ)}$ by assigning to each $\wt\gamma = \sum_i a_i \epsilon_i$ the 
function $\delta_{\wt\gamma}$ where $\delta_{\wt\gamma}(j) = \sum_{i \in j} 
a_i$. Furthermore, we set $\CC^{(\JJ)}_0 = \vspan{\delta_j - \delta_{j'} \mid 
j,j' \in \JJ}$, and introduce the linear map $h: \CC_0^{(\JJ)} \to \CC$, 
$h(\delta_j-\delta_{j'}) = \# \interval{j,j'} - 1$. If $\wt\gamma,\wt\nu \in 
\lie h^*$ and $\wt\nu - \wt\gamma$ is in the root lattice of $\gl(\infty)$, we 
put $h(\wt\nu,\wt\gamma) = h(\delta_{\wt\nu - \wt\gamma})$.

In an analogous way we assign to every $\wt\gamma \in \lie h^\circ$ a map 
$\delta^\infty_{\wt\gamma}: \JJ_{\infty} \to \CC$, define a linear map 
$h_{\infty}: \CC^{(\JJ_\infty)}_0 \to \CC$, and a function $h_\infty(\wt\nu,
\wt\gamma) = h_\infty(\delta^\infty_{\wt\nu - \wt\gamma})$ defined on pairs 
$\wt\gamma, \wt\nu$ such that $\wt\nu - \wt\gamma$ is in the root lattice. If 
$\wt\nu \succ \wt\gamma$ then $0 \leq h_\infty(\wt\nu,\wt\gamma) \leq 
h(\wt\nu,\wt\gamma)$.

\begin{defn}
Given $\wt\nu,\wt\gamma \in \lie h^*$ we write $\wt\nu \ggcurly \wt\gamma$ 
if $\wt\nu = \wt\gamma$ or $\wt\nu - \wt\gamma$ can be written as a sum of 
roots $\epsilon_i - \epsilon_j$ with $h_\infty(\epsilon_i,\epsilon_j) > 0$.
Equivalently, $\wt\nu \ggcurly \wt\gamma$ if and only if $\wt\nu \succ 
\wt\gamma$ and $h_\infty(\wt\nu,\wt\gamma) = h(\wt\nu,\wt\gamma)$.
\end{defn}

\subsection{Pieri sets and filtrations} 

\begin{defn}
If $\mathbf F = \bigwedge^d \VV, \bigwedge^A \VV$ we define 
$\Pieri{\wt\lambda}{\mathbf F}$ to be the set of all $\wt\gamma \in \supp 
\mathbf F$ such that $\wt\lambda + \wt\gamma$ is $\lie b_{\preceq}$-dominant, 
i.e., such that $\wt\lambda_i + \wt\gamma_i \geq \wt\lambda_{i'} + 
\wt\gamma_{i'}$ whenever $i \prec i'$.
If $\mathbf F = \SS^d\VV$ we define $\Pieri{\wt\lambda}{\mathbf F}$ to be the 
set of all $\wt\gamma \in \supp \mathbf F$ such that $\wt\lambda_i \geq \wt
\lambda_{i'} + \wt\gamma_{i'}$ whenever $i \prec i'$.
\end{defn}
We point out that $\wt\lambda + \wt\gamma$ is $\lie b_{\preceq}$-dominant whenever
$\wt\gamma \in \Pieri{\wt\lambda}{\SS^d\VV}$. Since 
$\Pieri{\wt\lambda}{\mathbf F}$ is a set of weights, it is a poset with the 
order induced by the order $\preceq$ of $\lie h^*$.

We now give a different characterization of the Pieri posets. Recall that
the support of $\SS^d\VV$ consists of weights $\epsilon_I$ with $I 
= \{i_1 \preceq i_2 \preceq \cdots \preceq i_d\}$, while that of $\bigwedge^d 
\VV$ is similar but with strict inequalities. If $I$ is any such sequence and
$X \subset \II$ then $I \cap X$ denotes the subsequence $\{i_{r_1} \preceq 
\cdots \preceq i_{r_s} \mid i_{r_k} \in X\}$. Finally, the support of 
$\bigwedge^A \VV$ consists of weights of the form $\epsilon_A - \epsilon_I + 
\epsilon_{I'}$ with $I \subset A, I' \subset B$ and $\#I = \#I'$. 

\begin{lem}
\label{lem:Pieri-varia}
The poset $\Pieri{\wt\lambda}{\mathbf F}$ is locally finite, and for any 
$\wt\mu \in \Pieri{\wt\lambda}{\mathbf F}$ the map $\wt\gamma \in 
\Pieri{\wt\lambda}{\mathbf F} \mapsto f_{\wt\gamma - \wt\mu} \in \CC^{(\JJ)}_0$ 
is injective. 
\end{lem}
\begin{proof}
Suppose first $\mathbf F = \SS^d\VV, \bigwedge^d \VV$.
Consider the set $\JJ^d$ with the product order. This is a locally finite poset 
since $\JJ$ is isomorphic to a subset of $\ZZ$.
Given $\epsilon_I \in \Pieri{\wt\lambda}{\mathbf F}$ with $I = 
(i_1, \ldots, i_d)$, we have $\pi_{\JJ}(I) = (\pi_{\JJ}(i_1), \ldots, 
\pi_{\JJ}(i_d)) \in \JJ^d$. If $\mathbf F = \SS^d\VV$ we recover $I$ from 
$\pi_\JJ(I)$ by taking $i_k$ to be the first element in $\pi_\JJ(i_k)$. If 
$\mathbf F = \bigwedge^d \VV$ we recover $I$ from $J$ as follows: let 
$i_1$ to be the first element of $\pi_\JJ(i_1)$, and once $i_k$ has been 
determined let $i_{k+1}$ to be the minimal element in $\pi_\JJ(i_{k+1})$ that 
is larger than $i_k$. Thus the map $I \mapsto \pi_\JJ(I)$ is injective in 
both cases. Notice also that $\epsilon_{I} - \epsilon_{I'} \preceq 0$ if and 
only if $I$ is less than $I'$ seen as elements of $\II^d$ with the product 
order, which happens if and only if $\pi_\JJ(I) < \pi_{\JJ}(I')$. It follows 
that the map $I \mapsto \pi_\JJ(I)$ is monotone, so $\Pieri{\wt\lambda}{\mathbf 
F}$ is isomorphic to a subposet of $\JJ^d$ and therefore locally finite. 

Consider now the map that sends $J \in \JJ^d$ to $\delta_J = \sum_{k=1}^d 
\delta_{j_k} \in \CC^{(\JJ)}$. If we restrict the map to increasing sequences 
in $\JJ^d$ then we get an injective map, and so $ \epsilon_I \mapsto 
f_{\wt\gamma} = \delta_{\pi_\JJ(I)}$ is an injective map. It follows that, if 
we fix $\wt\mu \in \Pieri{\wt\lambda}{\mathbf F}$, the map $\wt\gamma 
\mapsto f_{\wt\gamma} - f_{\wt\mu}$ is also injective.

Now suppose $\mathbf F = \bigwedge^A \VV$ with $A \subset \II$ and $B = \II 
\setminus A$ infinite sets. By definition $\epsilon_A \in \Pieri{\wt\lambda}
{\mathbf F}$. Also, $\wt\gamma = \epsilon_A - \epsilon_I + \epsilon_{I'} \in 
\supp \mathbf F$ if and only if $I \cap j$ is a terminal subset of $A \cap j$ 
and $I' \cap j$ is an initial subset of $B \cap j$ for all $j \in \JJ$. The 
assignation $\wt\gamma \in \Pieri{\wt\lambda}{\mathbf F} \mapsto (\pi_{\JJ}(I), 
\pi_{\JJ}(I')) \in \JJ^{(\ZZ_{>0})} \times \JJ^{(\ZZ_{>0})}$ is once again 
injective and monotone, so $\Pieri{\wt\lambda}{\mathbf F}$ is locally finite. 
We can assume $\wt\mu = \epsilon_A$, and in this case $f_{\wt\gamma - \wt\mu} = 
\delta_{\pi_{\JJ}(I)} - \delta_{\pi_\JJ(I')}$, from which we can recover 
$(\pi_\JJ(I), \pi_\JJ(I'))$, and hence $\wt\gamma$. This completes the proof.
\end{proof}

\begin{defn}
Given $\wt\gamma,\wt\nu \in \Pieri{\wt\lambda}{\mathbf F}$ we write $\wt\gamma 
\equiv_\infty \wt\nu$ if and only if $h_\infty(\wt\gamma, \wt\nu) = 0$, and set 
$\LL_\infty = \Pieri{\wt\lambda}{\mathbf F}/\equiv_\infty$. We denote by 
$\leq_\infty$ the linear order on $\LL_\infty$ induced by the order opposite of 
$\preceq$.
\end{defn}

For a fixed $\wt\mu \in \Pieri{\wt\lambda}{\mathbf F}$ the map $\wt\gamma \in 
\Pieri{\wt\lambda}{\mathbf F} \mapsto h_\infty(\wt\mu,\wt\gamma) \in \ZZ$ is 
injective and monotone, so $\LL_\infty$ is isomorphic to a subset of $\ZZ$. 
If $\mathbf F$ is a $\lie b_{\preceq}$-highest weight module then we can take
$\wt\mu$ to be its highest weight. Notice that in this case the class of $\wt
\mu$ in $\LL_\infty$ is the minimum of $\LL_\infty$. 

\begin{defn}
The \emph{linkage filtration} of $\mathbf M$ is $\mathcal F = \{F_\ll \mathbf 
M\}_{\ll \in \LL_\infty}$ where $F_\ll M$ is the $\gl(\infty)$-module generated 
by $\{v_{\wt\lambda} \otimes e_{\wt\gamma} \mid \wt\gamma \in \ll' \leq_\infty 
\ll\}$. If $\LL_\infty$ has a minimum $\ell$ we set $F_{\ll-1}\mathbf M = 
\{0\}$, and if $\LL_\infty$ has a maximum $\ll'$ we set $F_{\ll'+1} \mathbf M = 
\mathbf M$.
\end{defn}

\subsection{The statements}
We introduce a technical condition.
\begin{defn}
Let $\pi = \pi_\infty \circ \pi_\JJ: \II \to \JJ_\infty$. We say $A$ is 
\emph{$\JJ_\infty$-initial} if for every $j \in \JJ_\infty$ the set $A \cap 
\pi^{-1}(j)$ is an initial subset of $\pi^{-1}(j)$.
\end{defn}
In the following statements we assume that $\mathbf F$ is either $\SS^d\VV, 
\bigwedge^d \VV$ or a Fock module $\bigwedge^A \VV$ where $A$ is 
$\JJ_\infty$-initial.

\begin{thm*}[\textbf A]
The representation $\mathbf M$ is multiplicity free, has a composition series, 
and every simple constituent is of the form $\mathbf L(\wt\lambda + \wt\gamma)$ 
for some $\wt\gamma \in \Pieri{\wt\lambda}{\mathbf F}$. Furthermore, a simple 
constituent $\mathbf L(\wt\lambda + \wt\nu)$ is linked to $\mathbf L(\wt\lambda 
+ \wt\gamma)$ if and only if $\wt\nu \ggcurly \wt\gamma$.
\end{thm*}

\begin{cor*}[\textbf A]
The length of $\mathbf M$ is equal to the cardinality of $\Pieri{\wt\lambda}
{\mathbf F}$. Moreover, if $\mathbf F =\SS^d\VV, \bigwedge^d \VV$ then the length
of $\mathbf M$ is finite if and only if $\JJ$ is finite, while if $\mathbf F
= \bigwedge^A \VV$ then $\mathbf M$ has finite length if and only if $\mathbf 
L(\wt\lambda)$ is of the form $\mathbf D_n \otimes \mathbf L'$, with $\mathbf 
L'$ a simple tensor module.
\end{cor*}

Recall that a module $\mathbf N$ is said to be \emph{uniform} if any two 
nontrivial submodules have nontrivial intersection. In particular a uniform 
module is indecomposable. 
\begin{thm*}[\textbf B]
The module $\mathbf M$ is either indecomposable or semisimple. More precisely:
\begin{enumerate}
  \item If $\mathbf F$ is a $\lie b$-highest weight module and 
    $\LL_{\infty}$ is trivial, then $\mathbf M$ is semisimple.

  \item If $\mathbf F$ is a $\lie b$-highest weight module and 
    $\LL_{\infty}$ is not trivial, then $\mathbf M$ is indecomposable.

  \item If $\mathbf F$ is not a $\lie b$-highest weight module, then 
    $\mathbf M$ is uniform and has no socle.
\end{enumerate}
\end{thm*}

\begin{thm*}[\textbf C] 
\begin{enumerate}
  \item The module $\mathbf M$ has an exhaustive socle filtration if and only 
    if $\mathbf F$ is a $\lie b$-highest weight module. If this 
    happens then the socle filtration and the linkage filtration coincide. 

  \item The module $\mathbf M$ has a separated radical filtration if and only 
    if $\LL_\infty$ has a maximum. If this happens then the radical filtration 
    and the linkage filtration coincide. 

  \item The module $\mathbf M$ has finite Loewy length if and only if 
    $\LL_\infty$ is finite, and in this case $\mathcal L\ell(\mathbf M) = 
    \#\LL_\infty - 1$. Furthermore, $\mathbf M$ is rigid.
\end{enumerate}
\end{thm*}

We say that a linear order $\preceq$ on $\II$ is $(\wt\lambda, \mathbf 
F_*)$-compatible if it is $(-\wt\lambda,\mathbf F)$-compatible. We also put
$\Pieri{\wt\lambda}{\mathbf F_*} = -\Pieri{-\wt\lambda}{\mathbf F}$, and define
$\LL_\infty$ as before.

\begin{thm*}[\textbf D]
Let $\mathbf F = \SS^d \VV$ or $\mathbf F = \bigwedge^d \VV$, and
let $\mathbf M' := \mathbf L_{\lie b}(\wt\lambda) \otimes \mathbf 
F_*$ for a Borel subalgebra arising from a $(\wt\lambda, \mathbf 
F_*)$-compatible linear order $\preceq$ on $\II$. Then all above statements 
hold for $\mathbf M'$ if one replaces in them $\mathbf F$ with $\mathbf F_*$.
\end{thm*}

Before we proceed to the proofs we illustrate the above theorems with examples.

\subsection{Examples}
In 3.4.1--3.4.4 the module $\mathbf F$ equals $\SS^d\VV$ or $\bigwedge^d \VV$ 
for a fixed $d \in \ZZ_{>0}$.

\subsubsection{Tensor modules}
Let $\II = \ZZ_{>0} \sqcup \ZZ_{<0}$, and $\wt\lambda = \sum_{i \in \II} 
\wt\lambda_i \epsilon_i$ with $\wt\lambda_i = 0$ for all $i \in \interval{n+1,
-m-1}$ with $n,m \in \ZZ_{\geq0}$. Then $\lambda = (\wt\lambda_1, \ldots, 
\wt\lambda_n)$ and $\mu = (-\wt\lambda_{-1}, \ldots, -\wt\lambda_{-m})$ are 
nonempty partitions, and $\mathbf L(\wt\lambda)$ is isomorphic to the tensor 
module $\mathbf V^{\lambda;\mu}$ from \cite{PS11b}. The set $\JJ$ is a finite 
set of cardinality at most $n+1+m$, and $\JJ_{\infty}$ has exactly two elements.
The module $\mathbf M$ is semisimple if and only if $\mu = \emptyset$ and
simple if and only if $\mathbf L(\wt\lambda) = \CC$.

If $\mathbf F = \SS^d \VV$ then $\mathbf F$ has highest weight $d \epsilon_1$, 
and $\epsilon_I \in \supp \mathbf F$ belongs to $\Pieri{\wt\lambda}{\mathbf F}$ 
if and only if $I \cap j = \{i_1 \preceq \cdots \preceq i_r\}$ with $i_r = \min 
j$ and $r \in 
\ZZ_{\geq 0}$. If $\mathbf F = \bigwedge^d \VV$ its highest weight is $\epsilon_1 + \cdots + \epsilon_d$, and $\epsilon_I \in \supp \mathbf F$ belongs to 
$\Pieri{\wt\lambda}{\mathbf F}$ if and only if for each $j \in \JJ$ the $I \cap 
j \subset j$ is an initial subset. 

Set $q = \min\{\sum_{i=1}^m \lambda_{-i} - \lambda_{-i-1}, d\}$ when $\mathbf F 
= \SS^d \VV$ and $q = \min\{-\sum_{i=1}^m \lambda_{-i},d\}$ when $\mathbf F = 
\bigwedge^d \VV$. Then the set $\LL_\infty$ is isomorphic to $\interval{0,q}$, 
and through this identification the class of $\epsilon_I \in \Pieri{\wt\lambda}
{\mathbf F}$ in $\LL_\infty$ is $\# (I \cap \ZZ_{<0})$. In particular the Loewy 
length of $\mathbf M$ is $q$ and $\mathbf M$ is rigid. The $r$-th layer of the 
socle filtration is isomorphic to the direct sum of all $\mathbf L(\wt\lambda + 
\epsilon_I)$ with $\# I\cap \ZZ_{<0} = r$, and $\mathbf L(\wt\lambda + 
\epsilon_I)$ is linked to $\mathbf L(\wt\lambda + \epsilon_J)$ if and only if 
$I \setminus J \subset \ZZ_{>0}$ and $J \setminus I \subset \ZZ_{<0}$. In 
particular $\soc \mathbf M$ is simple if and only if $\lambda = \emptyset$.

\subsubsection{The case $\II \subset \ZZ$}
Suppose $\II = \ZZ_{>0}$. Then $\wt\lambda = \sum_{i=1}^\infty \wt\lambda_i 
\epsilon_i$ with $\wt\lambda_i \geq \wt\lambda_{i+1}$. Thus $\JJ$ is order 
isomorphic to a subset of $\ZZ_{>0}$. Moreover $\JJ_\infty$ has a single 
element, so  $\mathbf M$ is semisimple. The module $\mathbf M$ has finite 
length if and only if $\JJ$ is finite, i.e., if and only if $\wt\lambda_i = a$ 
for some $a \in \ZZ$ and all $i \gg 0$. Suppose now that $\II = \ZZ_{<0}$ or 
$\II = \ZZ$. Then $\JJ$ cannot have a minimal element and the module 
$\mathbf M$ is uniform, has infinite length and its socle is trivial. However, 
every nontrivial quotient of $\mathbf M$ is semisimple. All such quotients have 
finite length if $\JJ$ has a maximum, and have infinite length if $\JJ$ is does 
not have a maximum.

\subsubsection{Infinite staircase module}
In this example we take $\II = \bigsqcup_{n \in \ZZ_{> 0}} \ZZ_{\geq 0}$ and 
$\wt \lambda_{k,n} = -n$. Then $\JJ = \JJ_{\infty} \cong \ZZ_{>0}$. For 
$\mathbf F = \SS^d\VV$ the poset $\Pieri{\wt\lambda}{\mathbf F}$ is 
isomorphic to the poset of sequences $(i_{1}, i_{k+2}, \ldots, i_{d})$ 
such that $1 = i_1 = \cdots = i_k < i_{k+1} < i_{k+2} < \cdots < i_d$ for some 
$k \leq d$, with the lexicographic order. The correspondence assigns to each 
sequence the weight $k \epsilon_{0,1} + \epsilon_{0,i_{k+1}} + \cdots + 
\epsilon_{0,i_{d}}$. For $\mathbf F = \bigwedge^d \VV$ the poset 
$\Pieri{\wt\lambda}{\mathbf F}$ is isomorphic to the poset of all 
non-decreasing sequences $(i_1, \ldots, i_d) \in \ZZ_{\geq 0}^d$ with the 
lexicographic order, and the correspondence assigns to a sequence the weight
$\epsilon_{0,i_1} + \cdots + \epsilon_{0,i_d}$.

In both cases $\Pieri{\wt\lambda}{\mathbf F}$ has a minimum, namely 
$(1,1,\ldots, 1)$. In addition, $\Pieri{\wt\lambda}{\mathbf F}$ is a lattice 
and every maximal chain ending at a sequence has the same length, so the poset
$\Pieri{\wt\lambda}{\mathbf F}$ is ranked; the rank function coincides with 
$\wt\gamma \mapsto h_\infty(\wt\mu, \wt\gamma)$, where $\wt\mu$ is the highest 
weight of $\mathbf F$. The linkage filtration on $\mathbf M$
is thus infinite and coincides with the socle filtration. A simple $\mathbf 
L(\wt\lambda + \wt\gamma)$ appears in the $r$-th layer of the filtration if and 
only if $r$ is the rank of the sequence corresponding to $\wt\gamma$. In 
particular the socle of $\mathbf M$ is simple and isomorphic to $\mathbf 
L(\wt\lambda + d \epsilon_{0,1})$ for $\mathbf F = \SS^d\VV$ and to $\mathbf 
L(\wt\lambda + \epsilon_{0,1} + \cdots + \epsilon_{0,d})$ for $\mathbf F = 
\bigwedge^d \VV$. Since the poset $\Pieri{\wt\lambda}{\mathbf F}$ has finitely 
many elements of a fixed rank, each layer of the socle filtration has finitely 
many simple summands.

\subsubsection{Double-infinite staircase module}
Take $\II = \bigsqcup_{n \in \ZZ} \ZZ_{>0}$ and $\wt \lambda_{k,n} = -n$. Then 
$\JJ = \JJ_{\infty} \cong \ZZ$. Here $\Pieri{\wt\lambda}{\mathbf F}$ can be 
identified with the poset of strictly increasing sequences of integers of length 
$d$ when $\mathbf F = \SS^d\VV$, and with the poset of non-decreasing sequences 
of integers of length $d$ when $\mathbf F =\bigwedge^d \VV$. In both cases the 
poset is locally ranked so we can define a rank function on $\Pieri{\wt\lambda}
{\mathbf F}$ taking values in $\ZZ$. The socle and radical filtration of 
$\mathbf M$ are trivial. However, the linkage filtration has semisimple layers, 
and two simples appear in the same layer if and only if the rank of the 
corresponding sequences is equal. This implies in particular that every layer 
has infinitely many simple summands.

\subsubsection{Dual Fock modules}
Let $\mathbf F = \bigwedge^A \VV$ and $\wt\lambda = -\epsilon_C$ for infinite 
sets $A, C \subset \II$, and put $B = \II \setminus A, D = \II \setminus C$. To 
describe the $(\wt\lambda,\mathbf F)$-compatible order on $\II$ we 
let $W = A \cap D, X = B \cap D, Y = A \cap C$ and $Z = B \cap C$. By the 
definition of a $(\wt\lambda,\mathbf F)$-compatible order, $X$ and $Z$ are 
isomorphic to final subsets of $\ZZ_{<0}$, $Y$ and $W$ are isomorphic to 
initial subsets of $\ZZ_{\geq 0}$, and $\II = W \sqcup X \sqcup Y \sqcup Z$. 

Identifying $X \sqcup Y$ and $Z \sqcup W$ with the corresponding subsets of 
$\ZZ$, we write the elements of $\II$ as $(r,i)$ with $r \in \ZZ$ and $i =0,1$. 
The simple constituents of $\mathbf M = \mathbf L(\wt\lambda) \otimes \mathbf 
F$ have highest weights $\wt\mu_r$, where
\begin{align*}
  \wt\mu_r = \sum_{i \geq r} \epsilon_{(-i,0)} -  \epsilon_{(i,1)}
\end{align*}
for all $r$ such that this sum makes sense. For example, if $r = 1$ we recover
the simple constituent with highest weight $\epsilon_A - \epsilon_C$. The module 
$\mathbf M$ always has infinite length, and is semisimple if and only if $\# Y 
\sqcup Z < \infty$. If $\# Y \sqcup Z = \infty$ then $\LL_\infty$ can be 
identified as an ordered set with $\{r \in \ZZ \mid \wt\mu_r \in 
\Pieri{\wt\lambda}{\mathbf F}\}$, and the $r$-th layer of the linkage 
filtration is isomorphic to $\mathbf L(\wt\mu_r)$. In the case 
where $\# A \triangle C < \infty$, the existence of the linkage filtration and 
the description of its layers are due to Serganova, see 
\cite{Serganova21}*{Theorem 1.1}.

\subsubsection{Modules with finite $\JJ$}
\label{ex:modules_with_finite_jj_}
Set again $\mathbf F = \bigwedge^A \VV$ and let us consider the more general 
case where $\wt\lambda$ takes finitely many different values. This is 
equivalent to $\JJ$ being finite and hence $\JJ_\infty$ being also finite. 
Suppose $\JJ_{\infty} = \{j_1 \prec j_2  \prec \cdots \prec j_r\}$, and let 
$A_k = A \cap j_k, B_k = B \cap j_k$. We can always assume that $A_k, B_k$ is 
either infinite or empty. The definition of a $(\wt\lambda, 
\mathbf F)$-compatible order shows that $j_k = A_k \sqcup B_k$. The set 
$\Pieri{\wt\lambda}{\mathbf F}$ consists of those weights $\epsilon_A + 
\epsilon_I - \epsilon_J$ such that $J \cap A_k$ is a terminal subset of $A_k$, 
$I \cap B_k$ is an initial subset of $B_k$, and for each $k$ at least one of $J 
\cap A_k$ and $I \cap B_k$ is empty. 

The module $\mathbf M$ has a socle if and only if $\mathbf F$ is a $\lie 
b$-highest weight module, i.e., if and only if $A$ is an initial subset of 
$\II$. In this case the linkage filtration coincides with the socle filtration 
of $\mathbf M$. 

Suppose at least two elements of $\JJ$ are infinite classes. Then there exist 
two distinct infinite classes $j,j' \in \JJ$ such that $A \cap j$ and $B \cap 
j'$ are also infinite; for simplicity we assume that $j \prec j'$. If $I$ is a 
terminal subset of $A \cap j$ and $J$ is an initial subset of $B \cap j'$ then 
$\epsilon_{A \setminus I \cup J} \in \Pieri{\wt\lambda}{\mathbf F}$ and
\[
  h_\infty(\epsilon_A, \epsilon_A - \epsilon_I + \epsilon_J)
    = h_\infty(\epsilon_I,\epsilon_J)
    = (\# I) (\# \interval{j,j'})
\] 
Thus if $\JJ$ has two infinite classes the length of $\mathbf M$ and its Loewy
length are both infinite. 

On the other hand if $\JJ$ has exactly one infinite class then $\mathbf 
L(\wt\lambda)$ is of the form $\mathbf D_n \otimes \mathbf V^{\lambda, 
\mu}$ with $\lambda, \mu$ finite partitions. In this case we can assume $\II  
= \interval{r} \sqcup \ZZ \sqcup \interval{s}$, where $r$ is the number of 
distinct entries of $\lambda$ and $s$ the number of distinct entries in $\mu$.
Up to finite exchanges, the only infinite subsets of $\II$ with infinite 
complement are either initial or terminal. Since the order on $\II$ is $(\wt
\lambda, \mathbf F)$-compatible it follows that $A$ is an initial subset and 
$\mathbf F$ is a highest weight module. Without loss of generality we can 
assume $A = \interval{r} \cup \ZZ_{\leq 0}$.
The set $\LL_\infty$ is isomorphic to $\interval{0,r+s}$, with the class of a 
weight $\epsilon_I$ equal to $2 \#(I \cap \interval s) + \max I \cap \ZZ$. The 
socle of $\mathbf M$ is simple and isomorphic to $\mathbf L(\wt\lambda + 
\epsilon_A)$, and the radical of $\mathbf M$ is also simple and isomorphic to 
$\mathbf L(\wt\lambda + \epsilon_{A'})$ with $A' = \ZZ_{\leq r-s} \cup 
\interval s$.

\subsubsection{A non-example}
We finish this section with an example in which our results do not apply. Let 
$\II = \ZZ_{>0}, \wt\lambda = \sum_{i=1}^\infty (-i)\epsilon_i$ and $\mathbf F
= \bigwedge^A \VV$ with $A = 2\ZZ_{>0}$. Then $\mathbf F$ is not a $\lie 
b$-highest weight module since $\II$ has no infinite initial subsets 
other than itself. Also $\JJ \cong \ZZ_{>0}$ and $\# \JJ_\infty = 1$ so $A$
is not $\JJ_\infty$-initial. As we show in \ref{ex:non-ex} and
\ref{ex:non-example-bis} below, the module $\mathbf M$ has a 
simple subquotient isomorphic to $\mathbf L(\wt\lambda + \epsilon_B)$, and 
since $\epsilon_B$ is not in the support of $\mathbf F$ Theorem $A$ does not 
hold.

\bigskip

The rest of the paper is devoted to proving of Theorems A--D.


\section{Transition maps}
We start the build-up for the proofs by addressing some key issues in the 
finite-dimensional theory.
Given a finite set $Y \subset \II$ we set $V_Y := \vspan{e_i \mid i \in Y}, 
V_Y^* := \vspan{e_i^* \mid i \in Y}$. Then $\gl(Y) := V_Y^* \otimes V_Y 
\subset \gl(\infty)$ is isomorphic to $\gl(\# Y)$. For a subspace $W \subset 
\gl(\infty)$ we write $W_Y = W \cap \gl(Y)$. In particular, $\lie h_Y$ is a 
Cartan subalgebra of $\gl(Y)$ and all weights $\lambda \in \lie h^*_Y$ are 
assumed to be integral. We will assume that $\II$ has a total order 
$\preceq$ so $\lie b_Y = \lie b_{\preceq} \cap \gl(Y)$ is a Borel subalgebra 
of $\gl(Y)$, and we will refer to $\lie b_Y$-dominant weights simply as 
dominant weights. 

\subsection{Weight multiplicity-free representations and Pieri rules}
Let $Y \subset \II$ be a finite subset. It is well known that the modules 
$\SS^dV_Y$ with $d \geq 0$ and $\bigwedge^d V_Y$ with $0 \leq d \leq \# Y$ are 
weight multiplicity-free, and every other weight multiplicity free 
$\gl(Y)$-module can be obtained by dualizing and tensoring these modules with 
one-dimensional modules. In what follows, a skew Young diagram is called a 
vertical strip if it contains at most one box in each row, and a horizontal 
strip if it contains at most one box in each column.

Let $\lambda = \sum_{i \in Y} \lambda_i \epsilon_i \in \lie h_Y^*$ be a 
dominant weight, and let $F = \SS^dV_Y$ or $F = \bigwedge^d V_Y$. 
The tensor product $L_Y (\lambda) \otimes F$ is semisimple and multiplicity 
free, i.e., every simple submodule appears exactly once in its direct sum 
decomposition. We denote by $\Pieri{\lambda}{F}$ the set of weights $\gamma \in 
\supp F$ such that $\lambda + \gamma$ is a highest weight in $L_Y(\lambda) 
\otimes F$. In this language, the classical Pieri rules \cite{FH91}*{
formulas (6.8) and (6.9)} state that $\gamma \in \Pieri{\lambda}{\bigwedge^d 
V_Y}$ if and only if $\lambda + \gamma$ is dominant, while $\gamma \in 
\Pieri{\lambda}{\SS^dV_Y}$ if and only if $\gamma_{i+1} \leq 
\lambda_i - \lambda_{i+1}$. If $\lambda_i \geq 0$ for all $i \in Y$ then 
$\lambda$ can be read as a partition with $\# Y$ parts, which we identify it 
with its Young diagram. Moreover $\gamma \in \supp F$ belongs to 
$\Pieri{\lambda}{F}$ if and only if the skew diagram $(\lambda + \gamma) / 
\lambda$ is a horizontal strip for $F = \SS^dV_Y$, or a vertical strip 
for $F = \bigwedge^d V_Y$ (see for example \cite{Fulton97}*{section 2.2 (4) and 
(5)}). In both cases, given $\nu,\gamma \in \Pieri{\lambda}{F}$ we have $\nu 
\succ \gamma$ if and only if the $k$-th box of the strip $(\lambda + \nu) 
/\lambda$ is located above and to the right of the $k$-th box of the strip 
$(\lambda + \gamma) / \lambda$ for all $k \in \interval d$. 

\begin{lem}
\label{lem:pieri-lattice}
Let $F$ be a simple weight multiplicity-free $\gl(Y)$-module and $\lambda$ be a
dominant weight. 
\begin{enumerate}
  \item The poset $\Pieri{\lambda}{F}$ is a lattice.

  \item For every $\gamma \in \supp F$ the set $\{\mu \in \Pieri{\lambda}{F} 
    \mid \mu \succeq \gamma\}$ has a least element, which we denote $\gamma'$.

  \item Assume $\gamma \in \Pieri{\lambda}{F}$ and $\nu \in \supp F$ are such 
    that $\nu = \gamma + \epsilon_I - \epsilon_J$ and $\max I \prec \min J$. 
    Then $\nu' = \gamma + \epsilon_{I'} - \epsilon_{J'}$ and $\max I' \preceq 
    \max I \prec \min J \preceq \min J'$.
\end{enumerate}
\end{lem}
\begin{proof}
It is enough to prove the results when $\lambda_i \geq 0$ and $F$ equals 
$\SS^d V_Y$ or $\bigwedge^d V_Y$. 

Let $\nu, \gamma \in \Pieri{\lambda}{F}$. The skew diagram obtained by
superimposing the strips $(\lambda + \nu)/\lambda$ and $(\lambda + \gamma) / 
\lambda$ is again a horizontal or vertical strip, with at most $2d$ boxes. We 
let $\tilde \mu$ to be the strip obtained by taking the first $d$ boxes in this 
superposition, counting from bottom to top and from right to left. Then 
$\lambda \cup \tilde\mu = \lambda + \mu$ for some $\mu \in \Pieri{\lambda}
{F_Y}$, and by construction any vertical or horizontal strip $(\lambda + 
\sigma)/\sigma$ will have its $k$-th box below the $k$-th box of $\tilde \mu$
if and only if it is below the $k$-th box of both $(\lambda + \nu)/\lambda$
and $(\lambda + \gamma)/\lambda$. Thus every pair of elements in 
$\Pieri{\lambda}{F}$ has a maximal lower bound. Since the highest weight of $F$ 
is the maximum of $\Pieri{\lambda}{F}$, \cite{StanleyI}*{3.3.1 Proposition}
implies $\Pieri{\lambda}{F}$ is a lattice and item $(i)$ is proved. 

Every subset of a finite lattice has a unique maximal lower bound. Denoting 
$\gamma'$ the maximal lower bound of $\{\mu \in \Pieri{\lambda}{F} \mid \mu 
\succeq \gamma\}$, we see that $\gamma' \succeq \gamma$ since $\gamma$ is tautologically a lower bound for this set. 

We will prove item $(iii)$ by induction on $r = \# I = \#J$. First assume $\nu 
= \gamma + \epsilon_i - \epsilon_j$ with $\lambda_i > \lambda_j$. By the 
previous item there is a maximal $j' \preceq j$ such that $(\lambda + \gamma - 
\epsilon_{j'})/\lambda$ is a horizontal or vertical strip and $\gamma - 
\epsilon_{j'} \succeq \gamma - \epsilon_j$. There is also a minimal $i' 
\succeq i$ such that $(\gamma - \epsilon_{j'} + \epsilon_{i'})/\lambda$ is a 
horizontal or vertical strip and $\gamma - \epsilon_{j'} + \epsilon_{i'} \succeq 
\gamma - \epsilon_{j'} + \epsilon_i \succeq \gamma - \epsilon_j + \epsilon_i$. 
Since $\nu' = \gamma + \epsilon_{i'} - \epsilon_{j'}$ the base case is complete.
Now suppose $I = (i_1, \ldots, i_r)$ and $J = (j_1, \ldots, j_r)$. Set $\tilde I
= (i_1, \ldots, i_{r-1}), \tilde J = (j_1, \ldots, j_{r-1})$ and $\mu = \gamma 
+ \epsilon_{\tilde I} - \epsilon_{\tilde J}$. Then using the induction
hypothesis and the base case we have
\begin{align*}
  \nu' 
    &= (\mu' + \epsilon_{i_r} - \epsilon_{j_r})'
    = \mu' + \epsilon_{i'_r} - \epsilon_{j'_r}
    = \gamma + \epsilon_{\tilde I'} - \epsilon_{\tilde J'} + \epsilon_{i'_1} - 
    \epsilon_{j'_r}.
\end{align*} 
Set $I' = \tilde I' \cup \{i'_r\} $ and $J' = \tilde J' \cup \{j'_r\}$. By the 
induction hypothesis $\max \tilde I' \preceq i_{r-1} \preceq i_r$ and $\min J
= j_1 \preceq \min \tilde J'$, and by the base step $i'_r \preceq i_r$ and $j_1 
\preceq j_r \preceq j'_r$. Combining both observations we see that $\max I' 
\preceq i_r = \max I$ and $j_1 = \min J \preceq \min J'$.
\end{proof}

\subsection{Transition maps}
\label{ss:transition_maps}
Let $X \subset Y \subset \II$ be fixed finite sets, and let $\lambda \in \lie 
h_Y^*$ be a dominant weight. The simple module $L_Y(\lambda)$ is 
a $\gl(X)$-module by restriction, and the $\gl(X)$-module generated by its 
highest weight vector is isomorphic to $L_X(\lambda_X)$. We thus get an 
injective $\gl(X)$-homomorphism $\iota: L_X(\lambda_X) \to L_Y(\lambda)$. 
Let $F_Y$ be a weight multiplicity-free $\gl(Y)$-module, and let $F_X$ be a 
simple $\gl(X)$-submodule of $F_Y$ (not necessarily generated by the highest 
weight vector of $F_Y$). Then $F_X$ is also weight multiplicity free, and 
thus we get a $\gl(X)$-homomorphism $\tau: F_X \to F_Y$. 

Set $\phi(X,Y) := \iota \otimes \tau : L_X(\lambda_X) \otimes F_X \to 
L_Y(\lambda) \otimes F_Y$. By the Pieri rule this homomorphism decomposes as
\begin{align*}
  \xymatrix{
     L_X(\lambda_X) \otimes F_X \ar[r]^-{\cong}
      \ar[d]_-{\phi(X,Y)}
      & \displaystyle \bigoplus_{ \gamma \in \Pieri{\lambda_X}{F_X}} 
         L_X(\lambda_X +  \gamma) \ar[d]^-{\phi(X,Y,\gamma, \nu)} \\
     L_Y(\lambda) \otimes F_Y \ar[r]^-\cong
      & \displaystyle \bigoplus_{\nu \in \Pieri{\lambda}{F_Y}} 
         L_Y(\lambda +  \nu).
  }
\end{align*}
Since the domain of each homomorphism $\phi(X,Y,\gamma, \nu)$ is a simple 
$\gl(X)$-module, $\phi(X,Y,\gamma, \nu)$ is either zero or injective. Notice 
that $\phi(X,Y,\gamma, \nu) \neq 0$ if and only if the $\gl(Y)$-submodule 
generated by $\phi(X,Y)(L_X(\lambda + \gamma))$ contains $L_Y(\lambda + 
\nu)$. We will devote the rest of this section to give necessary and sufficient
conditions for $\map{\gamma}{\nu}{X}{Y}$ to be nonzero.

\subsection{Necessary conditions}
\label{ss:necessary_conditions}
Recall we have fixed $X \subset Y \subset \II$. In addition we suppose that 
there is a decomposition $Y = X^- \sqcup Z \sqcup X^+$ where $Z = Y \setminus 
X$. We also assume $(F_X, F_Y) = (\SS^dV_X, \SS^dV_Y)$ or $(F_X,F_Y) = 
(\bigwedge^d V_X, \bigwedge^d V_Y)$, with the map $\tau: F_X \to 
F_Y$ induced by the natural inclusion $V_X \to V_Y$, and fix a dominant 
$\gl(Y)$-weight $\lambda$. For any $X \subset Y$ we denote by $\lambda_X$ 
its restriction to $\lie h_X$. We have an injection $\lie h_X^* \to \lie h_Y^*$
mapping $\epsilon_i \in \lie h^*_X$ to $\epsilon_i \in \lie h_Y^*$, which 
allows us to consider a weight  $\gamma \in \lie h^*_X$ as a weight in $\lie 
h^*_Y$.

\begin{lem}
\label{lem:nonzero-necessary}
Let $\gamma \in \Pieri{\lambda_X}{F_X}$ and $\nu 
\in \Pieri{\lambda}{F_Y}$. If $\map{\gamma}{\nu}{X}{Y} \neq 0$ then there exist 
sequences $I \subset X^- \sqcup Z$ and $J \subset X^+$ such that $\nu - \gamma
= \epsilon_I - \epsilon_J$.
\end{lem}
\begin{proof}
We proceed by induction on $\# Z$. Suppose first that $Z = \{z\}$, and that
$\map{\gamma}{\nu}{X}{Y} \neq 0$. Then $L_Y(\lambda + \nu)$ has a 
$\gl(X)$-submodule isomorphic to $L_X(\lambda_X + \gamma)$. The 
classical branching rule \cite{Zhelobenko73}*{Theorem 2, section 66} (see also 
\cite{FH91}*{Exercise 6.12}) tells us that $\lambda_i + \nu_i \geq 
\lambda_i + \gamma_i \geq \lambda_{i+1} + \nu_{i+1}$ for all $i \in X^-$ and 
$\lambda_{i-1} + \nu_{i-1} \geq \lambda_i + \gamma_i \geq \lambda_i + \nu_i$ 
for all $i \in X^+$ (here the successor and predecessor are taken in $Y$). In 
particular $\nu_i \geq \gamma_i$ if $i \in X^-$ and $\gamma_i \geq \nu_{i}$ if 
$i \in X^+$, so
\begin{align*}
  \nu - \gamma = \sum_{i \in X^-} (\nu_i - \gamma_i) \epsilon_i
    + \nu_z \epsilon_z - \sum_{i \in X^+} (\gamma_i - \nu_i) \epsilon_i.
\end{align*}
This proves the base case.

For the induction step, assume $Z = Z' \sqcup \{z\}$ with $Z'$ nonempty. Set 
$Y' = X^- \sqcup Z' \sqcup X^+$ and $F_{Y'} = \SS^dV_Y$ or $F_{Y'} = 
\bigwedge^d V_{Y'}$. Since $\map{\gamma}{\nu}{X}{Y} \neq 0$ there exists $\mu 
\in \Pieri{\lambda_{Y'}}{F_{Y'}}$ such that $\map{\gamma}{\mu}{X}{Y'} \neq 0$ 
and $\map{\mu}{\nu}{Y'}{Y} \neq 0$. Thus by the induction hypothesis
\begin{align*}
\nu - \gamma 
  = (\nu - \mu) + (\mu - \gamma) 
  = (\epsilon_{K} - \epsilon_L) + (\epsilon_{I'} - \epsilon_{J'})  
\end{align*}
with $K \subset X^- \sqcup Z, I' \subset X^- \sqcup Z'$ and $L, J' \subset 
X^+$, and the proof is complete.
\end{proof}

\subsection{Nonzero transition maps}
The following is a partial converse of Lemma \ref{lem:nonzero-necessary}.

\begin{prop}
\label{prop:nonzero-sufficient}
Let $\gamma \in \Pieri{\lambda_X}{F_X}$ and $\nu \in \Pieri{\lambda}{F_Y}$ and let
$I, J$ be sequences such that $\nu - \gamma = \epsilon_I - \epsilon_J$. 
\begin{enumerate}
  \item Suppose $X^+ = \emptyset$. Then $\map{\gamma}{\nu}{X}{Y} \neq 0$ if and 
    only if $I = J = \emptyset$.
  \item Suppose $X^- = \emptyset$. Then $\map{\gamma}{\nu}{X}{Y} \neq 0$ if and 
    only if $I \subset Z$ and $J \subset X^+$.
  \item Suppose $X^+ \neq \emptyset \neq X^-$, $\lambda(Z) = \{a\}$ 
    with $a > \lambda_{\min X^+}$, and $\#Z \geq d$. Then 
    $\map{\gamma}{\nu}{X}{Y} \neq 0$ if and only if $I \subset X^- \sqcup Z$ 
    and $J \subset X^+$. 
\end{enumerate}
\end{prop}
We will prove each item separately. The first one is easy.

\begin{proof}[Proof of item $(i)$ of Proposition \ref{prop:nonzero-sufficient}]
The condition is necessary by Lemma \ref{lem:nonzero-necessary}, and as a 
consequence $\phi(X,Y,\gamma,\nu) = 0$ whenever $\nu \neq \gamma$. On 
the other hand, since $Y = X \sqcup Z$ the definition of $\Pieri{\lambda}{F_Y}$ 
implies that $\gamma \in \Pieri{\lambda}{F_Y}$. Finally, since at least one of 
the transition maps $\map{\gamma}{\nu}{X}{Y}$ must be nonzero, it follows that 
$\map{\gamma}{\gamma}{X}{Y} \neq 0$.
\end{proof}

\subsection{Highest weight vectors and the case $Y = Z \sqcup X$}
Our strategy to prove items $(ii)$ and $(iii)$ of Proposition 
\ref{prop:nonzero-sufficient} is to take the highest weight vector 
$v_{\lambda + \gamma}^X \in L_X(\lambda_X + \gamma) \subset L_X(\lambda_X) 
\otimes F_X$ and show that the $\gl(Y)$-module generated by $\phi(X,Y)
(v^X_{\lambda + \gamma})$ has a $\gl(Y)$-submodule isomorphic to $L_Y(\lambda + 
\nu)$. For this we will need some technical results on highest weight vectors.

Let $\mu \in \supp F_X$. We denote by $M_X(\succ \mu)$ the sum of all 
submodules $L_X(\lambda_X + \gamma)$ with $\gamma \in \Pieri{\lambda}{F_X}$ and 
$\gamma \succ \mu$. The submodules $M_X(\succeq \mu)$ and $M_X(\npreceq \mu)$ 
are defined analogously.  
\begin{lem}
\label{lem:hw-support}
Let $\gamma \in \Pieri{\lambda}{F_X}$ and let $v_{\lambda + \gamma}^X$ 
be the highest weight vector of $L_X(\lambda_X +\gamma) \subset L_X(\lambda_X) 
\otimes F_X$. Then 
  \begin{align*}
      v_{\lambda +\gamma}^X \equiv \sum_{\mu \succeq \eta \succeq \gamma}
      w_{\lambda + \gamma - \eta} \otimes e_{\eta}
      \mod M_X(\npreceq \mu).
  \end{align*}
Furthermore, $w_{\lambda + \gamma - \eta} \neq 0$ for all $\eta$, and their 
linear span is an $\lie n_X$-submodule of $L_X(\lambda_X)$ isomorphic to 
$\bigoplus_{\mu \succeq \eta \succeq \gamma} (F_X^*)_{-\eta}$.
\end{lem}
\begin{proof}
We start with the case where $\mu$ is the highest weight vector of $F_X$, and 
so $M_X(\npreceq \mu) = 0$. Then $v^X_{\lambda+\gamma} = \sum_{\eta \in \supp 
F_X} w_{\lambda + \gamma - \eta} \otimes e_{\eta}$, with $w_{\lambda + \gamma 
- \eta} \in L_X(\lambda_X)$ of weight $\lambda_X + \gamma - \eta$. Since 
$\lambda_X$ is the highest weight of this module we see that $w_{\lambda + 
\gamma - \eta} = 0$ unless $\gamma \preceq \eta$, so $v_{\lambda + \gamma}^X = 
\sum_{\mu \succeq \eta \succeq \gamma} w_{\lambda + \gamma - \eta} \otimes e_{\eta}$.

By standard Lie theory we have isomorphisms
\begin{align*}
  (L_X(\lambda) \otimes F_X)^{\lie n_X}
    &\cong \hom_{\lie n_X}(F_X^*,L_X(\lambda))
  &(L_X(\lambda) \otimes F_X)^{\lie n_X}
    &\cong \hom_{\lie n_X}(L_X(\lambda)^*,F_X).
\end{align*}
The first one sends $v_{\lambda+\gamma}^X$ to the $\lie n_X$-homomorphism $\psi$
given by $\psi(e_{\eta}^*) = w_{\lambda+\gamma-\eta}$, and so the span of the 
vectors $w_{\lambda + \gamma - \eta}$ is the image of the $\lie n_X$-submodule 
$\bigoplus_{\mu \succeq \eta \succeq \gamma} (F_X^*)_{-\eta}$ by $\psi$. The
statement will be proved once we show that the restriction of $\psi$ to 
$\bigoplus_{\mu \succeq \eta \succeq \gamma} (F_X^*)_{-\eta}$ is injective, 
which is equivalent to the fact that 
$\psi(e_{\gamma}^*) = w_\lambda$ is nonzero. Now, the second isomorphism 
sends $v_{\lambda+ \gamma}^X$ to the unique $\lie n_X$-homomorphism defined by 
$v_{\lambda}^* \mapsto 
v_{\lambda}^*(w_{\lambda}) e_{\gamma}$. Since $v_{\lambda + \gamma}^X \neq 0$, 
it must be the case that $v_{\lambda}^*(w_{\lambda}) \neq 0$ and hence 
$w_{\lambda} \neq 0$. This completes the proof when $\mu$ is the highest weight 
of $F_X$.

We now proceed with the general case. A descending induction argument, again 
starting with the highest weight $\mu$ of $F_X$, shows that $M_X(\succeq \mu)$
is generated by $v_{\lambda}^X \otimes e_{\mu}$ as a $\gl(X)$-module. It 
follows that if $\eta \npreceq \mu$ then $v^X_{\lambda}\otimes e_{\eta} \in 
M_X(\npreceq \mu)$. Now, since $L_X(\lambda) \otimes F_X$ is generated as an 
$\lie n^-_X$-module by the set $\{v_{\lambda}^X \otimes e_{\eta} \mid \eta \in 
\supp F_X\}$, the quotient $M / M_X(\npreceq \mu)$ is generated as an $\lie 
n^-_X$-module by $\{v_{\lambda}^X \otimes e_{\eta} \mid \eta \preceq \mu\}$. If 
we act by any element of $\lie n^-_X$ on any such generator, we obtain a linear 
combination of vectors $\sum_{\mu \succeq \eta \succeq \gamma} w_{\lambda + 
\gamma - \eta} \otimes e_{\eta}$ for some $w_{\lambda + \gamma - \eta}$. 
In particular $v_{\lambda +\gamma}^X \equiv \sum_{\mu \succeq \eta \succeq 
\gamma} w_{\lambda + \gamma - \eta} \otimes e_{\eta} \mod M_X(\npreceq \mu)$, 
and the second part of the statement follows from the previous case.
\end{proof}

\begin{proof}[Proof of item $(ii)$ of Proposition \ref{prop:nonzero-sufficient}]
That the condition is necessary follows from Lemma \ref{lem:nonzero-necessary}. 
To see that it is sufficient we proceed by induction on $\#Z$. Suppose that $Z =
\{z\}$ and that $\nu - \gamma = \sum_{i \in X} a_j (\epsilon_z - \epsilon_i)$ 
for some $a_i \in \ZZ_{\geq 0}$. Let $E = \prod_{j \in X} E_{z,i}^{a_i} \in 
U(\gl(Y))_{\nu - \gamma}$; this product is well defined since all the factors 
commute. By Lemma \ref{lem:hw-support}
\begin{align*}
  \phi(X,Y)(v_{\lambda +\gamma}^X) 
    &\equiv \sum_{\nu \succeq \eta \succeq \gamma} 
      w_{\lambda + \gamma - \eta} \otimes e_{\eta}
      \mod M_Y(\npreceq \nu).
\end{align*}
Applying $E$ to the sum on the right hand side of the congruence we obtain a 
sum of terms of the form $E_{\sigma} w_{\lambda + \gamma - \eta} \otimes 
E_{\tau} e_{\eta}$, where $\sigma + \tau = \nu - \gamma$. Notice that $\sigma 
\neq 0$ implies $\sigma_z \neq 0$.

By weight considerations, $E_{\sigma} w_{\lambda + \gamma - \eta} = 0$ unless 
$\sigma + \lambda + \gamma - \eta \preceq \lambda$, or equivalently $ \sigma + 
\gamma \preceq  \eta$. On the other hand, $E_{\sigma} w_{\lambda + \gamma - 
\eta} \equiv 0$ modulo $M(\npreceq \nu)$ unless $\tau + \eta \preceq \nu$, 
which implies $\eta \preceq \nu - \tau = \gamma + \sigma$. Thus $\sigma = \eta 
- \lambda_X$. Since $\supp \eta - \lambda_X \subset X$, it must be the case 
that $\sigma = 0$. In conclusion
\begin{align*}
  E \phi(X,Y)(v_{\lambda +\gamma}^X) 
    &\equiv w_{\lambda } \otimes E e_{\gamma}
    \equiv c w_{\lambda} \otimes e_{\nu}
      \mod M_Y(\npreceq \nu).
\end{align*}
for some nonzero constant $c$. This vector is a nonzero multiple of 
$v_{\lambda} \otimes e_{\nu}$ which generates $L_Y(\lambda + \nu)$
modulo $M_Y(\npreceq \nu)$, and so the direct summand $L_Y(\lambda + \nu)$ is 
contained in the $\gl(Y)$-module generated by $\phi(X,Y)(L_X(\lambda + \gamma)
)$. This shows that $\phi(X,Y,\gamma,\nu) \neq 0$ and completes the proof of 
the base case.

We now proceed with the inductive step. Suppose $Z = \{z\} \sqcup Z'$ and 
$\nu - \gamma = \sum_{i \in Z, j \in X} a_{i,j} (\epsilon_i - \epsilon_j)$. 
Set $Y' = Z' \sqcup X$ and $\mu = \gamma + \sum_{i \in Z', j \in X} a_{i,j} 
(\epsilon_i - \epsilon_j)$. By the induction hypothesis 
$\map{\gamma}{\mu}{X}{Y'}$ and $\map{\mu}{\nu}{Y'}{Y} \neq 0$, 
which implies $\map{\gamma}{\nu}{X}{Y} \neq 0$. 
\end{proof}

\subsection{The case $X^+ \neq \emptyset \neq X^-$}
\label{ss:nonzero-ggcurly}
For proving item $(iii)$ of Proposition \ref{prop:nonzero-sufficient} we need 
some preparation. 
Suppose $A \subset X$ and let $C = X \setminus A$. Then $\gl(A) \oplus \gl(C)$
is a Lie subalgebra of $\gl(X)$ and any $\gl(X)$-module is a $\gl(A) \oplus 
\gl(C)$-module by restriction. In particular there are decompositions of 
$\gl(A) \oplus \gl(C)$-modules
\begin{align*}
  \SS^dV_X
    &=\bigoplus_{k+\ell = d} \SS^k V_A \boxtimes \SS^\ell V_C,
  &\bigwedge^d V_X
    &=\bigoplus_{k+\ell = d} \bigwedge^k V_A \boxtimes \bigwedge^\ell V_C.
\end{align*}
Thus, if $\gamma \in \supp F_X$ there exist a unique weight multiplicity-free 
$\gl(A)$-module $F_A$ and a unique weight multiplicity-free $\gl(C)$-module 
$F_C$ such that $\gamma \in \supp F_A \boxtimes F_C$. Since $F_C$ is a trivial 
$\gl(A)$-module we have a $\gl(A)$-homomorphism $p_\gamma: F_A \otimes F_C \to 
F_A$ given by
\begin{align*}
p_\gamma(e_\eta) 
  &= \begin{cases}
    e_{\eta_A} & \mbox{ if } \eta_C = \gamma_C \\
    0          & \mbox{otherwise.}
  \end{cases}  
\end{align*}

Now let $\mu \in \lie h_X$ be a dominant weight. The highest weight 
vector $v_\mu^X \in L_X(\mu)$ is a $\gl(A) \oplus \gl(C)$-highest weight 
vector and generates the unique $\gl(A) \oplus \gl(C)$-submodule of $L_X(\mu)$ 
isomorphic to $L_A(\mu_A) \boxtimes L_C(\mu_C)$. Since $L_C(\mu_C)$ has a 
trivial $\gl(A)$-action there exists a unique $\gl(A)$-homomorphism $L_C(\mu_C) 
\to \CC$ sending the highest weight vector $v_{\mu_C}^C$ to $1$ and annihilating
any weight space of weight lower than $\mu_C$. We denote by $p_A$ the 
$\gl(A)$-homomorphism given by the composition $L_X(\mu) \to L_A(\mu_A) \otimes 
L_C(\mu_C) \to L_A(\mu_A)$. With these preliminaries in place we can prove the 
following lemma.

\begin{lem}
\label{lem:reduction}
Let $\gamma \in \Pieri{\lambda_X}{F_X}$ and $\nu \in \Pieri{\lambda}{F_Y}$ with
$\nu - \gamma = \epsilon_I - \epsilon_J$. Suppose $B \subset Y$ is such that 
$I,J \subset B$, and let $A = B \cap X$. If $\map{\gamma_A}{\nu_B}{A}{B} \neq 0$
then $\map{\gamma}{\nu}{X}{Y} \neq 0$.
\end{lem}
\begin{proof}
We now prove that the diagram 
\begin{align*}
\begin{tikzcd}[ampersand replacement=\&]
  L_X(\lambda_X+\gamma) \ar[r]
    \& L_X(\lambda_X) \otimes F_X \ar[r] \ar[d,"p_A \otimes p_\gamma"]
    \& L_Y(\lambda) \otimes F_Y \ar[r] \ar[d, "p_B \otimes p_\nu"]
    \& L_Y(\lambda + \nu) \ar[d,"p_B"] \\
  L_A(\lambda_A+\gamma_A) 
    \ar[u,"\iota_A"] \ar[r]
    \& L_A(\lambda_A) 
      \otimes F_A \ar[r]
    \& L_B(\lambda_B) 
      \otimes F_B \ar[r]
    \& L_B(\lambda_B+\nu_B) 
\end{tikzcd}    
\end{align*}
commutes. As a consequence we will have $\phi(A,B,\gamma_A, \nu_B) = p_B \circ 
\phi(X,Y,\gamma,\nu) \circ \iota_A$, and hence the statement.

To show that the left square commutes we need to prove that the highest 
weight vector $v_{\lambda_A+\gamma_A}^A \in L_A(\lambda_A + \gamma_A) \subset 
L_A(\lambda_A) \otimes F_A$ is mapped to itself by the up-right-down path in the
diagram. By definition $v_{\lambda_A+\gamma_A}^A$ is mapped to $v_{\lambda_A + 
\gamma}^X \in L_X(\lambda_X + \gamma)$, which is then mapped to itself in 
$L_X(\lambda) \otimes F_X$. Since $p_A$ is a $\gl(A)$-homomorphism, the vector 
$p_A(v^X_{\lambda_X + \gamma})$ is a $\gl(A)$-highest weight vector of weight 
$\lambda_A + \gamma_A$ and hence equal to $v_{\lambda_A + \gamma_A}^A$.

To see that the middle square commutes it is enough to check the commutativity 
of the squares
\begin{align*}
\begin{tikzcd}[ampersand replacement=\&]
\& (1)
    \& L_X(\lambda_X)  \ar[r] \ar[d]
    \& L_Y(\lambda)\ar[d]
    \& (2)
    \& F_X \ar[r] \ar[d]
    \& F_Y  \ar[d]\\
\&    
    \& L_A(\lambda_A) \ar[r]
    \& L_B(\lambda_B)  
    \&
    \& F_A \ar[r]
    \& F_B 
\end{tikzcd}    
\end{align*}
The proof that square $(2)$ commutes is just diagram chasing, so we focus 
on square $(1)$.
We denote by $U^-_X$ the enveloping algebra of $\lie n_X^-$, and use similar 
notation for the sets $Y, A, B$. The PBW-theorem implies that $U^-_X$ is free 
as a right $U^-_A$-modules, that $U^-_Y$ is free as a right $U^-_B$-module, and 
that we can choose a $U^-_B$-basis $\mathcal B$ of $U^-_Y$ such that $\mathcal 
B \cap U^-_X$ is a $U^-_A$-basis of $U^-_X$. Without loss of generality we can 
assume $1 \in \mathcal B$. If $f = f_1 f_2 \in U^-_X$ with $f_1 \in \mathcal B$ 
and $f_2$ then $\iota_\lambda(f_1f_2 v^X_\lambda) = f_1f_2 v^Y_\lambda$ by 
definition, while
\begin{align*}
  p_A(f_1f_2 v^X_\lambda) &=
    \begin{cases}
      f_2 v^A_\lambda & \mbox{ if } f_1 = 1;\\
      0               & \mbox{otherwise}.  
    \end{cases}
\end{align*}
A similar formula holds for $p_B$. Now the proof that the first square commutes 
also reduces to diagram chasing.

To prove that the right square commutes, we observe that the map $L_Y(\lambda) 
\otimes F_Y \to L_B(\lambda) \otimes F_B$ maps $v^Y_\lambda \otimes e_\eta$ to
$v^B_\lambda \otimes e_{\eta_B}$ if $\eta_{Y \setminus B} = \nu_{Y \setminus 
B}$, and to zero otherwise. It follows that $M_Y(\npreceq \nu)$ is mapped to 
$M_B(\npreceq \nu_B)$, so we can expand the square as follows
\begin{align*}
\begin{tikzcd}[ampersand replacement=\&]
  L_Y(\lambda) \otimes F_Y \ar[r] \ar[d]
    \& \frac{L_Y(\lambda) \otimes F_Y}{M_Y(\npreceq \nu)} \ar[r] \ar[d]
    \& L_Y(\lambda + \nu) \ar[d] \\
  L_B(\lambda_B) 
      \otimes F_B \ar[r]
    \& \frac{L_B(\lambda_B) \otimes F_B}{M_B(\npreceq \nu)} \ar[r]
    \& L_B(\lambda_B+\nu_B).
\end{tikzcd}    
\end{align*}
The left square of this expanded diagram commutes by definition. On the other 
hand, the weight space $\left(\frac{L_Y(\lambda) \otimes F_Y}{M_Y(\npreceq 
\nu)} \right)_{\lambda + \nu}$ has dimension $1$, which implies that 
$\frac{L_Y(\lambda) \otimes F_Y}{M_Y(\npreceq \nu)}$ contains exactly one copy
of the $\gl(B)$-module $L_B(\lambda_B+\nu_B)$, generated by 
$v^\lambda_X \otimes e_\nu$. The image of this vector through either path in 
the right-hand square is equal to the projection of $v^\lambda_B 
\otimes e_{\nu_B}$ to $L_B(\lambda_B + \nu_B)$, which is equal to the highest 
weight vector $v^B_{\lambda_B + \nu_B}$. The result is proved.
\end{proof}

We now establish a stronger result than the one stated in item $(iii)$ 
Proposition \ref{prop:nonzero-sufficient}.

\begin{prop}
\label{prop:gg-sufficient}
Suppose that $X^+ \neq \emptyset \neq X^-$, and that $\lambda(Z) = \{a\}$ with 
$a > \lambda_{\min X^+}$. 
Let $\gamma \in \Pieri{\lambda_X}{F_X}$ and $\nu \in \Pieri{\lambda}{F_Y}$, and
let $I\subset X^- \sqcup Z, J \subset X^+$ be sequences such that $\nu - 
\gamma = \epsilon_I - \epsilon_J$ and $\#I = \#J \leq \#Z$. Then 
$\map{\gamma}{\nu}{X}{Y} \neq 0$.
\end{prop}
\begin{proof}
The proof goes by induction on $\#Z$, and we begin with the base case $Z = 
\{z\}$. Suppose first $\nu = \gamma + \epsilon_z - \epsilon_j$ for some $j \in 
X^-$, and put $B = \{z,j\}, A = B \cap X = \{j\}$. By Lemma \ref{lem:reduction} 
$\map{\gamma}{\nu}{X}{Y} \neq 0$ whenever $\map{\gamma_A}{\nu_B}{A}{B} \neq 0$. 
Since $B = \{z\} \sqcup A$, the fact that 
$\map{\gamma_A}{\nu_B}{A}{B} \neq 0$ follows from item $(i)$ of 
Proposition \ref{prop:nonzero-sufficient}.

Suppose now that $\nu = \gamma + \epsilon_i - \epsilon_j$ for $i \in X^-$ and 
$j \in X^+$. Taking $B = \{i,z,j\}$ and $A = B \cap X = \{i,j\}$, a new 
application of Lemma \ref{lem:reduction} shows that it is enough to prove that
$\map{\gamma_A}{\nu_B}{A}{B} \neq 0$. Thus we may assume that $X = \{i,j\}$ and 
$Y = \{i,z,j\}$. Since $\gl(X)\cong \gl(2)$, Lemma \ref{lem:hw-support} reduces to the statement
\begin{align*}
v_{\lambda + \gamma}
  &= v_\lambda \otimes e_\gamma + \sum_{r = 1}^n a_r E_{j,i}^r v_\lambda \otimes
    e_{\gamma + r(\epsilon_i - \epsilon_j)}.
\end{align*}
Applying $E_{i,j}$ to this highest weight vector we obtain
\begin{align*}
  0 &= E_{i,j} v_{\lambda + \gamma} = (v_\lambda + a_1E_{i,j}E_{j,i} v_\lambda) 
      \otimes e_{\gamma + \epsilon_i - \epsilon_j} + \cdots \\
    &= (1 + a_1(\lambda_i - \lambda_j))v_\lambda \otimes e_{\gamma + \epsilon_i 
      - \epsilon_j} + \cdots,
\end{align*}
so $a_1 = -\frac{1}{\lambda_i - \lambda_j}$. Notice that $\lambda_i - \lambda_j 
\neq 0$ since $\lambda_i \geq \lambda_z > \lambda_j$.
On the other hand, Lemma \ref{lem:hw-support} implies
\begin{align*}
  \phi(X,Y)(v_{\lambda + \gamma}) 
    \equiv v_\lambda \otimes e_\gamma \mod M_Y(\npreceq \nu).
\end{align*}
Applying $E_{z,j}E_{i,z}$ to $\phi(X,Y)(v_{\lambda + \gamma})$, and using the 
fact that $E_{i,z} e_\gamma = 0$ since $\gamma_z = 0$, we have
\begin{align*}
  E_{z,j}E_{i,z}\phi(X,Y)(v_{\lambda + \gamma}) 
    \equiv
      - \frac{1}{\lambda_i - \lambda_j} E_{z,j}E_{i,z}E_{j,i} v_\lambda 
      \otimes e_\nu \mod M_Y(\npreceq \nu).
\end{align*}
Now
\begin{align*}
  E_{z,j}E_{i,z}E_{j,i} v_\lambda
    &= E_{z,j} (E_{j,i}E_{i,z} - E_{j,z})v_\lambda 
    = - E_{z,j}E_{j,z} v_\lambda
    = -(\lambda_j - \lambda_z)v_\lambda
\end{align*}
so
\begin{align*}
  E_{z,j}E_{i,z}\phi(X,Y)(v_{\lambda + \gamma}) 
    \equiv
      \frac{\lambda_j - \lambda_z}{\lambda_i - \lambda_j} 
        v_\lambda \otimes e_\nu \mod M_Y(\npreceq \nu).
\end{align*}
Since $v_\lambda \otimes e_\nu$ generates $L_Y(\lambda + \nu)$ modulo 
$M_Y(\npreceq \nu)$, it follows that $\map{\gamma}{\nu}{X}{Y} \neq 0$. This 
completes the proof of the base case.

Let us now proceed with the inductive step. Let $Z = Z' \sqcup \{z\}$, $I = 
\{i_1 \preceq \cdots \preceq i_r\}, J = \{j_1 \preceq \cdots \preceq j_r\}$, 
and assume that the statement holds for $Z'$. Set $Y' = X^- \sqcup Z' 
\sqcup X^+, I' = I \setminus 
\{i_r\}, J' = J \setminus \{j_r\}$ and $\mu = \gamma + \epsilon_{I'} - 
\epsilon_{J'}$. Then $\mu \in \Pieri{\lambda_{Y'}}{F_{X \cup Z'}}$ since 
$\lambda_i = a$ for all $i \in Z$. We can now apply the induction hypothesis to 
conclude that 
$\map{\gamma}{\mu}{X}{Y'} \neq 0$ and $\map{\mu}{\nu}{Y'}{Y} \neq 0$, which in 
turn implies $\map{\gamma}{\nu}{X}{Y} \neq 0$.
\end{proof}

\begin{proof}[Proof of item $(iii)$ of Proposition \ref{lem:nonzero-necessary}]
That the result is necessary follows once more from Lemma 
\ref{lem:nonzero-necessary}. Now by definition $\# I = \#J = d$, so if $\# Z 
\geq d$ we can invoke Proposition \ref{prop:gg-sufficient} to deduce the result.
\end{proof}

\subsection{Further nonzero transition maps}
In this subsection we assume $d \leq \#X$ and set $r = \# Z$, $(F_X, F_Y) 
= \left(\bigwedge^d V_X, \bigwedge^{d+r} V_Y \right)$. We have a nonzero 
$\gl(X)$-homomorphism $\tau: F_X \to F_Y$ given by $e_\mu \mapsto e_{\mu} 
\wedge e_{\epsilon_Z}$, and so we are in the context of subsection 
\ref{ss:transition_maps}. We now establish a result analogous to Proposition 
\ref{prop:nonzero-sufficient} for the transition maps associated to the 
map $\phi(X,Y): L_X(\lambda) \otimes F_X \to L_Y(\lambda) \otimes 
F_Y$.

\begin{prop}
\label{prop:nonzero-sufficient-dual}
Let $\gamma \in \Pieri{\lambda_X}{F_X}$ and $\nu \in \Pieri{\lambda}{F_Y}$ and 
let $I, J$ be sequences such that $\nu - \gamma = \epsilon_I - \epsilon_J$. 
\begin{enumerate}
  \item Suppose $X^- = \emptyset$. Then $\map{\gamma}{\nu}{X}{Y} \neq 0$ if and 
    only if $I = J = \emptyset$.
  \item Suppose $X^+ = \emptyset$. Then $\map{\gamma}{\nu}{X}{Y} \neq 0$ if and 
    only if $I \subset X^-$ and $J \subset Z$.
  \item Suppose $X^+ \neq \emptyset \neq X^-$, $\lambda(Z) = \{a\}$ 
    with $a < \lambda_{\max X^-}$, and $\#Z \geq d' = d - \#X$. Then 
    $\map{\gamma}{\nu}{X}{Y} \neq 0$ if and only if $I \subset X^-$ 
    and $J \subset Z \sqcup X^+$. 
\end{enumerate}
\end{prop}
\begin{proof}
It is possible to prove the statement by proving a series of intermediate 
results just as we did for Proposition \ref{prop:nonzero-sufficient}. Instead,  
we will employ an involution $\chi$ of $\gl(Y)$ such that twisting
$\phi(X,Y)$ by $\chi$ produces a map $\tilde \phi(X',Y)$ to which Proposition 
\ref{prop:nonzero-sufficient} can be applied. 

We denote by $\omega_0$ the longest element of the Weyl group $S_Y$ of 
$\gl(Y)$. Let $\chi$ be the automorphism of $\gl(Y)$ given by $\chi(E_{i,j}) = 
-E_{\omega_0(j), \omega_0(i)}$. Notice that $\chi(\lie b_Y) = \lie b_Y$. Then 
${}^\chi L_Y(\lambda) \cong L_Y(-\omega_0(\lambda))$, and the highest weight 
vector $v_\lambda \in L_Y(\lambda)$ equals the highest weight vector of 
${}^\chi L_Y(\lambda)$. Recall $F_Y = L_Y(\epsilon_I)$ where $I$ is the 
sequence formed by the first $d+r$ elements of $Y$. By definition $\omega_0(I)$ 
is the sequence formed by the last $d+r$ elements of $Y$, and
\begin{align*}
{}^\chi F_Y 
    &\cong L_Y(-\epsilon_{\omega_0(I)})
    \cong L_Y(-\epsilon_Y) \otimes L_Y(\epsilon_{Y \setminus \omega_0(I)})
    \cong L_Y(-\epsilon_Y) \otimes \bigwedge^{d'} V_Y 
\end{align*}
where $d' = \# Y - d - r = \#X - d$. Set $\tilde F_Y = \bigwedge^{d'} V_Y$ and 
$\tilde \lambda = -\omega_0(\lambda) - \epsilon_Y$. Since twisting by $\chi$ 
commutes with tensor products, there is an isomorphism
\begin{align*}
  {}^\chi(L_Y(\lambda) \otimes F_Y)
    &\cong L_Y(-\omega_0(\lambda)) \otimes L_Y(-\epsilon_Y) \otimes 
       \bigwedge^{d'} V_Y \cong L_Y(\tilde \lambda) \otimes \tilde F_Y.
\end{align*}
Furthermore, comparing the decompositions afforded by the Pieri rule, a summand
$L_Y(\lambda + \nu)$ with $\nu \in \Pieri{\lambda}{F_Y}$ is twisted by $\chi$ 
into a summand $L_Y(-\omega_0(\lambda + \nu))$. Writing $-\omega_0(\lambda + 
\nu) = \tilde \lambda + \tilde \nu$ we have $\tilde \nu = \epsilon_Y - 
\omega_0(\nu) \in \Pieri{\tilde \lambda}{\tilde F_Y}$.

Now set $X' = \omega_0(X) \subset Y$. Then $\chi(\gl(X')) = \gl(X)$ and every 
$\gl(X)$-module becomes a $\gl(X')$-module through $\chi$. Similar arguments 
show that ${}^\chi L_X(\lambda_X) \cong L_{X'}
(-\omega_0(\lambda)_{X'})$, with the highest weight vector of $L_X(\lambda_X)$ 
corresponding to the highest weight vector of $L_{X'}
(-\omega_0(\lambda)_{X'})$. Taking $J \subset X$ to be the sequence formed by 
the first $d$ elements
\begin{align*}
  {}^\chi F_X
    &\cong L_{X'}(-\epsilon_{\omega_0(J)})
    \cong L_Y(-\epsilon_{X'}) \otimes L_Y(\epsilon_{X' \setminus \omega_0(J)})
    \cong L_Y(-\epsilon_{X'}) \otimes \bigwedge^{d'} V_X. 
\end{align*}
Thus setting $\tilde F_X = \bigwedge^{d'} V_X$
\begin{align*}
  {}^\chi(L_X(\lambda_X) \otimes F_X)
    &\cong L_Y(-\omega_0(\lambda)_{X'}) \otimes L_{X'}(-\epsilon_{X'}) \otimes 
      L_{X'}(\epsilon_J) \cong L_X(\tilde \lambda_{X'}) \otimes \tilde F_Y.
\end{align*}
Again comparing decompositions we see that $\gamma \in \Pieri{\lambda_X}{F_X}$ 
if and only if
$\tilde \gamma = \epsilon_{X'} - \omega_0(\gamma) \in \Pieri{\tilde 
\lambda_{X'}}{\tilde F_{X'}}$. 

Since $\chi$ preserves $\lie b_Y$ and sends $\lie b_X$ to $\lie b_{X'}$, the 
highest weight vectors of $L_X(\lambda)$ and $L_Y(\lambda)$ remain highest 
weight vectors of the respective twisted modules ${}^\chi L_{X}(\lambda_X)$ and
${}^\chi L_Y(\lambda)$. Thus the $\chi$-twist of the map 
$\iota_{\lambda}$ is the map $\iota_{-\omega_0(\lambda)}: 
L_{X'}(-\omega_0(\lambda)_{X'}) \to L_Y(-\omega_0(\lambda))$. Direct inspection
shows that the $\chi$-twist of $\tau: F_X \to F_Y$ is the map 
$\iota_{\epsilon_Y} \otimes \tilde\tau: \CC_{\epsilon_X} \otimes \tilde F_X 
\to \CC_{\epsilon_X} \otimes \tilde F_Y$, where $\tilde \tau: \tilde F_X \to 
\tilde F_Y$ is the map induced by $V_{X'} \to V_Y$. It follows that twisting
the map $\phi(X,Y): L_X(\lambda_X) \otimes F_X \to L_Y(\lambda) \otimes F_Y$ 
by $\chi$ we get the map $\phi(X',Y): L_X(\tilde \lambda_X) \otimes \tilde F_X
\to L_Y(\tilde \lambda) \otimes \tilde F_Y$, and the transition map $\phi(X,Y,
\gamma,\nu)$ is twisted into the transition map $\map{\tilde\gamma}{\tilde\nu}
{X'}{Y}$. Thus $\phi(X,Y,\gamma,\nu) \neq 0$ if and only if $\map{\tilde\gamma}
{\tilde\nu}{X'}{Y} \neq 0$. Proposition \ref{prop:nonzero-sufficient} now 
implies the statement.
\end{proof}

\section{Exhaustions of tensor products of $\gl(\infty)$-modules}

We now start our study of tensor products of $\gl(\infty)$-modules. 
An \emph{exhaustion of $\II$} is a chain of finite subsets $\{X_n\}_{n\geq 1}$ 
such that $\II = \bigcup_{n\geq 1} X_n$. In this section we prove the existence 
of exhaustions such that every pair $X_n \subset X_{n+1}$ satisfies the 
hypotheses of Propositions \ref{prop:nonzero-sufficient} and 
\ref{prop:nonzero-sufficient-dual}, which allows us to determine exactly which 
transition maps $\map{\gamma}{\nu}{X_n}{X_{n+1}}$ are nonzero.
We fix a weight $\wt\lambda \in \lie h^*$,
and for any finite subset $X$ of $\II$ we denote by $\wt\lambda_X$ the 
restriction of $\wt\lambda$ to $\lie h_X$.

\subsection{Exhaustions of $\II$ for $\mathbf F = \SS^d\VV, \bigwedge^d \VV$}
\label{ss:tensor-setup}
For this subsection we assume that $\mathbf F$ is either $\SS^d\VV$ or 
$\bigwedge^d\VV$ for some $d \in \ZZ_{>0}$, and for every finite set $X 
\subset \II$ we take $F_X = \SS^dV_X$ or $F_X = \bigwedge^d V_X$ accordingly.
Whenever $X \subset Y$ we have maps $F_X \to F_Y$ induced by the natural maps 
$V_X \to V_Y$, and $\mathbf F = \varinjlim F_X$. 

We fix a $(\wt\lambda,\mathbf F)$-compatible order $\preceq$ on $\II$, and set 
$\lie b = \lie b_{\preceq}$. Given $\gamma \in \supp F_X$ we denote by 
$\gamma^Y$ the image of $\gamma$ through the embedding $\lie h_X^* \to \lie 
h_Y^*$, and by $\overline{\gamma}$ the image of $\gamma$ through 
the embedding $\lie h_X^* \to \lie h^*$.

\begin{defn}
We say a finite set $X \subset \II$ is a \emph{$d$-block} if the following 
conditions hold:
\begin{enumerate}
  \item $\pi_\JJ(X)$ is an interval in $\JJ$.
  \item If $j \in \pi_{\JJ}(X)$ and $j$ is finite, then $j \subset X$.
  \item If $j \in \pi_{\JJ}(X)$ and $j$ is infinite, then $j \cap X$ is an
    initial subset of $j$ and contains at least $d$ elements.
\end{enumerate}
\end{defn}

Recall from Lemma \ref{lem:Pieri-varia} that for every $\gamma \in \lie h_X^*$
there is a least element $\gamma' \in \Pieri{\wt\lambda_X}{F_X}$ such that
$\gamma' \succeq \gamma$. The following is immediate from the 
definitions and Lemma \ref{lem:Pieri-varia}.
\begin{lem}
\label{lem:d-compatible-pieri}
Let $X \subset \II$ be a $d$-block, and suppose $\gamma \in 
\Pieri{\wt\lambda_X}{F_X}$.
\begin{enumerate}
  \item If $\mathbf F = \bigwedge^d \VV$ then $\overline{\gamma} \in 
    \Pieri{\wt\lambda}{\mathbf F}$.
  \item If $\mathbf F = \SS^d\VV$ then there exists $Y \supset X$ such that 
    $\overline{(\gamma^Y)'} \in \Pieri{\wt\lambda}{\mathbf F}$. If $\min 
    \supp \gamma \prec \min X$ or $\min X = \min \II$ then we can take $Y = X$.
  \end{enumerate}
\end{lem}

\begin{defn}
Suppose $X \subset Y \subset \II$ are finite $d$-blocks. We will say that
$(X,Y)$ is a \emph{$d$-step} if the set $Z = Y \setminus X$ satisfies the 
following conditions: $(i)$ $\pi_\JJ(Z) = \{j\}$, and $(ii)$ if $j$ is infinite 
then $\# Z \geq d$. We call an exhaustion $\{X_n\}_{n \geq 1}$ of $\II$ a 
\emph{$d$-exhaustion} if each pair $(X_n,X_{n+1})$ is a $d$-step.
\end{defn}

Suppose $(X,Y)$ is a $d$-step with $Z = Y \setminus X$ and $\{j\} = 
\pi_\JJ(Z)$. 
Setting 
\begin{align*}
X^- &= \{x \in X \mid \pi_\JJ(x) \preceq j\}
  &\mbox{and}&
  &X^+ &= \{x \in X \mid \pi_\JJ(x) \succ j\},   
\end{align*}
we have $Y = X^- \sqcup Z \sqcup X^+$. It follows from the definitions that if 
$X^- \neq \emptyset \neq X^+$ then $j$ is an infinite class and every element 
of $Z$ has at least $d$ predecessors in $j$. The definition of a $d$-step is 
made so that we can apply Proposition \ref{prop:nonzero-sufficient}, and thus 
determine exactly which transition maps are nonzero in a $d$-step $(X,Y)$.

\begin{prop}
\label{prop:d-nonzero}
Let $(X,Y)$ be a $d$-step with $Y = X^- \sqcup Z \sqcup X^+$ and $\pi_\JJ(Z)
= \{j\}$. Let $\gamma \in \Pieri{\wt\lambda_X}{F_X}$ and $\nu \in 
\Pieri{\wt\lambda_Y}{F_Y}$, with $\overline{\nu} - \overline{\gamma} = 
\epsilon_I - \epsilon_J$ and $I,J$ of minimal length. 
\begin{enumerate}
  \item Suppose $X^+ = \emptyset$. Then $\map{\gamma}{\nu}{X}{Y} \neq 0$
    if and only if $I = J = \emptyset$.
  \item Suppose $X^- = \emptyset$. Then $\map{\gamma}{\nu}{X}{Y} \neq 0$ if and 
    only if $I \subset Z$ and $J \subset X^+$. Furthermore
    \begin{align*}
      h_\infty(\overline{\nu},\overline{\gamma}) =
        \begin{cases}
          h(\overline{\nu},\overline{\gamma}) & \mbox{ if $j$ is infinite; } \\
          0 &\mbox{ if $j$ is finite. }
        \end{cases}
    \end{align*}
  \item Suppose $X^- \neq \emptyset \neq X^+$. Then $\map{\gamma}{\nu}{X}{Y} 
    \neq 0$ if and only if $I \subset X^-$ and $J \subset X^+$. Furthermore 
    $h_\infty(\overline{\nu},\overline{\gamma}) = 
    h(\overline{\nu},\overline{\gamma})$.
\end{enumerate}
\end{prop}
\begin{proof}
Items $(i)$ and $(ii)$ follow from the corresponding items of Proposition 
\ref{prop:nonzero-sufficient}. For the third item, the proposition only 
guarantees that $I\subset X^- \sqcup Z$. Notice however that $I \subset \supp 
\nu$, and the Pieri rule guarantees that $I \cap j$ must be an 
initial subset of $j$ with at most $d$ elements. Since the first $d$ elements 
of $j$ belong to $X^-$, it follows that $I \subset X^-$.
\end{proof}

We now show that $d$-exhaustions exist. 

\begin{prop}
\label{prop:d-exhaustion}
There exists a $d$-exhaustion $\{X_n\}_{n \geq 1}$. If $\JJ$ has a minimum then 
the exhaustion can be chosen so that $\min X_n = \min \JJ$ for all $n$. 
\end{prop}
\begin{proof}
We first fix some notation. Whenever $j \in \JJ$ is an infinite class $j = 
\{i_1 \prec i_2 \prec \cdots\}$ we denote by $j^{(k)}$ the subset $\{i_{(k-1)d 
+ 1}, \ldots, i_{kd}\}$, so $j = j^{(1)} \sqcup j^{(2)} \sqcup \cdots$.

Since $\JJ$ is order-isomorphic to a subset of $\ZZ$, there exists a function
$\theta: \ZZ_{>0} \to \JJ$ such that $\theta(\interval{n})$ is an interval in 
$\JJ$ for all $n \in \ZZ_{>0}$ and $\theta^{-1}(j)$ is infinite for all $j$. 
Set $Y_0 = \emptyset$ and for each $m \in \ZZ_{>0}$ define $Y_{m}$ as 
follows: if $j = \theta(m)$ is finite then $Y_m = Y_{m-1} \cup j$, and if $j$ 
is infinite then $Y_m = Y_{m-1} \cup j^{(k+1)}$, where $k = \max \{ r \mid 
j^{(r)} \subset Y_{m-1} \}$. By construction each $Y_m$ is a $d$-block and  
either $Y_m = Y_{m+1}$ or $(Y_m,Y_{m+1})$ is a $d$-step. The exhaustion is 
obtained by setting $X_n$ to be the $n$-th different set in the sequence of the 
$Y_m$. If $\JJ$ has a minimum then we can take $\theta(1) = \min \JJ$, which 
guarantees that $\min \JJ \in X_1 \subset X_n$. 
\end{proof}

\begin{cor}
\label{cor:d-nonzero-maps}
Let $\{X_n\}_{n \geq 1}$ be a $d$-exhaustion. Suppose $\gamma \in 
\Pieri{\wt\lambda_{X_n}}{X_n}, \nu \in \Pieri{\wt\lambda_{X_m}}{X_m}$, and 
$\map{\gamma}{\nu}{X_n}{X_m} \neq 0$. Then there exist $\wt\mu \in \lie 
h^\circ$ with $h_\infty(\overline\nu,\overline\gamma) = h_\infty(\wt\mu,0) 
= h(\wt\mu,0)$,
and sequences $J \subset \supp \gamma$ and $I \subset X_m \setminus X_n$ such 
that $\overline \nu - \overline \gamma = \wt\mu + \epsilon_I - \epsilon_J$.
\end{cor}
\begin{proof}
The proof is an easy induction using Proposition \ref{prop:d-nonzero}.
\end{proof}

\subsection{Exhaustions of $\II$ for $\mathbf F = \bigwedge^A \VV$}
\label{ss:fock-setup}
Throughout this subsection $A \subset \II$ denotes an infinite set with 
infinite complement $B = \II \setminus A$, and $\mathbf F = \bigwedge^ A \VV$. 
For any finite set $X \subset \II$ we set $d(X) = \# A \cap X$ and $F_X = 
\bigwedge^{d(X)} V_X$. If $X \subset Y$ then there is a map $F_X \to F_Y$ given 
by $e_\eta \mapsto e_\eta \wedge e_{\epsilon_{K}}$ where $K = (Y \setminus X) 
\cap A$, and $\mathbf F = \varinjlim_X F_X$. We have maps $\lie h_X^* \to 
\lie h^*_Y$ given by $\gamma = \sum_{i \in X} \gamma_i \epsilon_i \mapsto 
\gamma^Y = \sum_{i \in X} \gamma_i \epsilon_i + \epsilon_K$, and 
$\lie h^*_X \to \lie h^*$ given by $\gamma \mapsto \overline{\gamma} = 
\sum_{i \in X} \gamma_i \epsilon_i + \epsilon_{A \setminus X}$. Notice that 
$\overline{\gamma}_Y = \gamma^Y$.

We fix on $\II$ a $(\wt\lambda, \mathbf F)$-compatible linear order $\preceq$
and set $\lie b = \lie b_{\preceq}$. In particular, for every $j \in \JJ$ the 
set $j \cap A$ is order-isomorphic to a subset of $\ZZ_{<0}$ and $j \cap B$ is 
order-isomorphic to a subset of $\ZZ_{>0}$. 

\begin{defn}
Let $A \subset \II$ be an infinite set with infinite complement $B = \II 
\setminus A$. Let $X$ be a finite subset of $\II$. We say $X$ is 
\emph{$(A,B)$-block} if the following conditions hold:
\begin{enumerate}
  \item $\pi_\JJ(X)$ is an interval in $\JJ$.
  \item If $j \in \pi_\JJ(X)$ is finite then $j \subset X$.
  \item If $j \in \pi_\JJ(X)$ is infinite then $j \cap A \cap X$ is a terminal  
    subset of $j \cap A$, and $j \cap B \cap X$ is an initial subset of $j 
    \cap B$.
\end{enumerate}
\end{defn}

The following Lemma is immediate from Lemma \ref{lem:Pieri-varia}.
\begin{lem}
\label{lem:AB-compatible-pieri}
Let $X \subset \II$ be an $(A,B)$-block, and let $\gamma \in 
\Pieri{\wt\lambda_X}{F_X}$. Then $\overline{\gamma} \in 
\Pieri{\wt\lambda}{\mathbf F}$.
\end{lem}

\begin{defn}
Suppose $X \subset Y \subset \II$ are $(A,B)$-blocks. We say 
that $(X,Y)$ is an \emph{$A$-step} if the set $Z = Y \setminus X$ satisfies the following 
conditions:
\begin{enumerate}
  \item $Z \subset j \cap A$ for some $j \in \JJ$;
  \item if $j$ is infinite then $\#Z \geq \# (X \cap B)$.
\end{enumerate}
We say $(X,Y)$ is a \emph{$B$-step} if the same conditions hold interchanging 
$A$ and $B$. An exhaustion $\{X_n\}_{n \geq 1}$ of $\II$ by finite sets is an 
\emph{$(A,B)$-exhaustion} if every pair $(X_n, X_{n+1})$ is either an $A$-step
or a $B$-step.
\end{defn}
Suppose $(X,Y)$ is a $B$-step. Then $d(X) = d(Y)$, and $(X,Y)$ is also a 
$d(X)$-step. In particular $Y = X^+ \sqcup Z \sqcup X^-$ just as in the 
previous subsection, and $F_X = \bigwedge^{d(X)} V_X, F_Y = \bigwedge^{d(X)}$. 
Proposition \ref{prop:d-nonzero} gives necessary and sufficient conditions for
a transition map to be nonzero.

Now suppose $(X,Y)$ is an $A$-step with $Z = Y \setminus X$ and $\pi_\JJ(Z) = 
\{j\}$. Then $d(Y) = d(X) + d(Z)$, so $F_X = \bigwedge^{d(X)} V_X, F_Y = 
\bigwedge^{d(X) + d(Z)} V_Y$, and the map $F_X \to F_Y$ is given by $-\wedge 
e_{\epsilon_Z}$. Setting
\begin{align*}
  X^- &= \{x \in X \mid \pi_\JJ(x) \prec j\}
    &\mbox{and}&
    & X^+ &= \{x \in X \mid \pi_\JJ(x) \succeq j\}
\end{align*}
we have $Y = X^- \sqcup Z \sqcup X^+$, and $X^- \neq \emptyset \neq X^+$ 
implies that $j$ is an infinite class. The next proposition shows that we know 
exactly which transition maps in an $A$-step are nonzero.
\begin{prop}
\label{prop:AB-nonzero}
Let $(X,Y)$ be an $A$-step with $Y = X^- \sqcup Z \sqcup X^+$ and $\pi_\JJ(Z)
= \{j\}$. Let $\gamma \in \Pieri{\wt\lambda_X}{F_X}$ and $\nu \in 
\Pieri{\wt\lambda}{F_Y}$, with $\overline{\nu} - \overline{\gamma} = \epsilon_I 
- \epsilon_J$ and $I,J$ of minimal length.
\begin{enumerate}
  \item Suppose $X^- = \emptyset$. Then $\map{\gamma}{\nu}{X}{Y} \neq 0$
    if and only if $I = J = \emptyset$.
  \item Suppose $X^+ = \emptyset$. Then $\map{\gamma}{\nu}{X}{Y} \neq 0$ if and 
    only if $I \subset X^-$ and $J \subset Z$. Furthermore 
    \begin{align*}
      h_\infty(\overline\nu,\overline\gamma)
        &= \begin{cases}
          h(\overline\nu, \overline\gamma)  & \mbox{ if $j$ is infinite };\\
          0 & \mbox{ if $j$ is finite}.
        \end{cases}
    \end{align*}
  \item Suppose $X^+ \neq \emptyset \neq X^-$. Then $\map{\gamma}{\nu}{X}{Y} 
    \neq 0$ if and only if $I \subset X^-$ and $J \subset Z \sqcup X^+$.
    Furthermore $h_\infty(\overline{\nu},\overline{\gamma}) = 
    h(\overline{\nu},\overline{\gamma})$.
\end{enumerate}
\end{prop}
\begin{proof}
These are all special cases of Proposition \ref{prop:nonzero-sufficient-dual}.
\end{proof}

\begin{prop}
\label{prop:AB-exhaustion}
There exists an $(A,B)$-exhaustion. If $A$ is an initial subset of $\II$ then 
$X_n$ can be chosen so that $\max A, \min B \in X_n$ for all $n$. 
\end{prop}
\begin{proof}
We consider the set $\tilde \JJ = \{j\cap A, j \cap B \mid j \in \JJ\}$ with 
the unique total order $\preceq$ such that $j \cap A \prec j \cap B$ for all $j 
\in \JJ$, and such that $j \cap B \prec j' \cap A$ for all $j,j' \in \JJ$ with
$j \prec j'$. Since $\JJ$ is isomorphic to a subset of $\ZZ$, so is $\tilde 
\JJ$. We now proceed as in the proof of Proposition \ref{prop:d-exhaustion}, using 
a map $\theta: \ZZ_{>0} \to \tilde \JJ$.
\end{proof}

\begin{cor}
\label{cor:AB-nonzero-maps}
Let $\{X_n\}_{n \geq 1}$ be an $(A,B)$-exhaustion. Suppose $\gamma \in 
\Pieri{\wt\lambda_{X_n}}{X_n}, \nu \in \Pieri{\wt\lambda_{X_m}}{X_m}$ and 
$\map{\gamma}{\nu}{X_n}{X_m} \neq 0$. Then there exist $\wt\mu \in \lie 
h^\circ$ with $h_\infty(\wt\mu,0) = h_\infty(\overline\nu,\overline\gamma)
= h(\wt\mu,0)$ and sequences $I \subset X_m \setminus (A \cup X_n), 
J \subset X_n \cap \supp \gamma, K \subset X_n \setminus \supp \gamma$ and 
$L \subset X_m \cap A$ such that $\overline \nu - \overline \gamma = \wt\mu + 
\epsilon_I - \epsilon_J + \epsilon_K - \epsilon_L$.
\end{cor}
\begin{proof}
The proof is by induction using Proposition \ref{prop:AB-nonzero}.
\end{proof}

\subsection{On sequences of weights}
We finish this section with a result that will be fundamental in the 
sequel. 

\begin{lem}
\label{lem:approximation}
Suppose $(X,Y)$ is a $d$-step, an $A$-step or a $B$-step. Suppose we have 
$\gamma, \sigma \in \Pieri{\wt\lambda_X}{F_X}$ and $\nu \in 
\Pieri{\wt\lambda_Y}{F_Y}$ with $\gamma - \sigma = \epsilon_I 
- \epsilon_J \succeq 0$ and $\map{\gamma}{\nu}{X}{Y}\neq 0$. 
Then there exists $\tau \in \Pieri{\wt\lambda_Y}{F_Y}$ such that $\nu - 
\tau = \epsilon_{I'} - \epsilon_{J'} \preceq \epsilon_{I} - \epsilon_{J}$,
$I' \subset X^-, J' \subset Z \sqcup X^+$ and $\map{\sigma}{\tau}{X}{Y} \neq 
0$. 
\end{lem}
\begin{proof}
Suppose $I = \{i_1 \preceq \cdots \preceq i_s\}, J = \{j_1 \preceq \cdots 
\preceq j_s\}$ and let $K, L$ be sequences of minimal length such that 
$\overline{\nu} - \overline{\gamma} = \epsilon_K - \epsilon_L$. 

Assume first that $(X,Y)$ is a $d$-step or a $B$-step. We have $K \subset X^-$ 
and $L \subset X^+$ by Proposition \ref{prop:d-nonzero}, and we set $R = L 
\cap I = \{i_{u_1} \preceq \cdots \preceq i_{u_r}\} \subset X^+$ and 
$S = \{j_{u_1} \preceq \cdots \preceq j_{u_r}\}$; since $i_k \prec j_k$ it 
follows that $S \subset X^+$. Now take $a,b$ such that
\begin{align*}
  J \cap X^+ 
    &= \{j_a \preceq \cdots \preceq j_s\}, 
  &I \cap X^- 
    &= \{i_1 \preceq \cdots \preceq i_b\},
\end{align*}
and set $P = \{i_a \preceq \cdots \preceq i_b\} \subset X^-$ and $Q = 
\{j_a \preceq \cdots \preceq j_b\} \subset X^+$. It follows that $b < u_1$, 
so $P$ and $R$ are disjoint subsequences of $I$, and $Q$ and $S$ are disjoint 
subsequences of $J$. We now set
\[
  \tau_0 
    = \sigma^Y + \epsilon_K - \epsilon_{(L \setminus R) 
      \cup S} + \epsilon_{P} - \epsilon_{Q}
\]
and $\tau := \tau_0'$. 
Since $K, P \subset X^-$ and $L, S, Q \subset X^+$, item $(iii)$ of Lemma 
\ref{lem:pieri-lattice} implies $\tau - \sigma^Y = \epsilon_U - \epsilon_V$
for sequences $U \subset X^-$ and $V \subset X^+$. Now item $(iii)$ of 
Proposition \ref{prop:nonzero-sufficient} implies $\map{\sigma}{\tau}{X}{Y}
\neq 0$. We also have
\begin{align*}
  \overline{\nu} - \overline{\tau} 
  \preceq
  \overline{\nu} - \overline{\tau_0}
    &= \overline{\gamma} + \epsilon_K - \epsilon_L - 
      \overline{\nu} - \epsilon_K + \epsilon_{(L \setminus R) \cup S} 
        - \epsilon_{P} + \epsilon_{Q} 
    = \epsilon_{I \setminus (P \cup R)} - \epsilon_{J \setminus (Q \cup S)}
\end{align*}
This concludes the proof of the first statement when $(X,Y)$ is a $d$-step or a 
$B$-step. If in particular $I \subset X^-$ and $J \subset X^+$ then $I = P, J 
= Q$ and $R,S = \emptyset$, and so $\nu = \tau$.

Now suppose $(X,Y)$ is an $A$-step, so $K \subset X^-$ and $L \subset Z \sqcup 
X^+$ by Proposition \ref{prop:AB-nonzero}. We take $a,b \in \interval s$ such 
that
\begin{align*}
    J \cap (X^+ \sqcup Z) 
    &= \{j_a \preceq \cdots \preceq j_s\} 
  &I \cap X^- 
    &= \{i_1 \preceq \cdots \preceq i_b\}
\end{align*}
and define $P, Q$ as above. We also set $S = J \cap X^- = \{j_{u_1} 
\preceq \cdots \preceq j_{u_r}\}$ and $R = \{i_{u_1} \preceq \cdots \preceq 
i_{u_r}\}$. Set
\[
  \tau_0
    = \sigma^Y + \epsilon_{(K \setminus S) \cup R} - \epsilon_{L} 
      + \epsilon_{P} - \epsilon_{Q}
\]
and $\tau = \tau_0'$. The rest of the proof proceeds as in the case 
where $(X,Y)$ is a $d$-step or a $B$-step.
\end{proof}

We now prove two results on sequences of weights $\{\mu_n\}_{n \geq r}$ 
with $\mu_n \in \Pieri{\wt\lambda_n}{F_n}$.

\begin{lem}
Let $\mathbf F = \bigwedge^A \VV$ with $A$ a $\JJ_\infty$-initial set, and let 
$(\mu_n)_{n \geq r}$ be a sequence with $\mu_n \in \Pieri{\wt\lambda_n}
{F_n}$ and $\overline{\mu_{n+1}} \succeq \overline{\mu_n}$. If the sequence 
$\overline{\mu_n}$ does not stabilize then there exists $n$ such that 
$\overline{\mu_{n+1}} \ggcurly \overline{\mu_n}$.
\end{lem}
\begin{proof}
Recall that $\pi: \II \to \JJ_\infty$ denotes the natural projection.
Let $\wt\mu \in \Pieri{\wt\lambda}{\mathbf F}$. For every $j \in \JJ_\infty$ 
we have a decomposition $(\supp \wt\mu) \cap \pi^{-1}(j) = A_{j} \sqcup I_{j}$, 
where $A_{j}$ is an initial subset of $\pi^{-1}(j)$, $I_{j}$ is finite, and 
the cardinality of $I_j$ is minimal among all such decompositions. 
Furthermore, since $\supp \wt\mu$ and $A$ differ in finitely many elements, 
$I_{j}$ is empty for all but finitely many $j \in \JJ_{\infty}$ and the set 
$\bigcup_{j \in \JJ_\infty} I_{j}$ is finite. For each $i \in I_{j}$ we 
set $t(i) = \# \interval{\max A_{j}, i} - 1$, and take $t(\wt\mu)$ to be the 
sum of all $t(i)$ with $i \in \bigcup_{j \in \JJ_\infty} I_{j}$. We will 
prove the statement by induction on $t = t(\overline{\mu_{r}})$.

Since the sequence $\overline{\mu_n}$ does not stabilize we can assume that 
$\overline{\mu_{r+1}} \succ \overline{\mu_{r}}$, and so $\overline{\mu_{r+1}} 
- \overline{\mu_{r}} = \epsilon_I - \epsilon_{I'}$. Suppose first that $t = 0$, 
and let $i = \min I, i' = \min I'$ and $j = \pi_{\JJ}(i), j' = \pi_{\JJ}(i')$. 
Since $i' \in \supp \overline{\mu_{r}}$ it follows that $i' \in A_{j'}$ and, 
since $A_{j'}$ is an initial subset of $\pi^{-1}(j')$, that $j \prec j'$. Thus 
$h_\infty(\overline{\mu_{r+1}}, \overline{\mu_{r}}) > 0$ and by Propositions 
\ref{prop:d-nonzero} and \ref{prop:AB-nonzero} this can only happen if 
$\overline{\mu_{r+1}} \ggcurly \overline{\mu_{r}}$.

Suppose now that the statement holds for all sequences starting with a weight 
$\mu$ such that $t(\overline{\mu}) < t$. As we have just seen, $\overline{
\mu_{r+1}} \ggcurly \overline{\mu_{r}}$ unless $h_\infty(\overline{\mu_{r+1}}, 
\overline{\mu_{r}}) = 0$. If $\overline{\mu_{r+1}} \ggcurly \overline{\mu_{r}}$ 
holds we are done. If $h_\infty(\overline{\mu_{r+1}}, \overline{\mu_{r}}) = 0$, 
taking $I = \{i_1 \prec \cdots \prec i_s\}$ and $I' = \{i'_1 \prec \cdots \prec 
i'_s\}$ it must be the case that $i'_k \prec i_k$ and $\pi(i'_k) = \pi(i_k)$. 
Consequently, $t(i'_k) < t(i_k)$ and hence $t(\overline{\mu_{r+1}}) < 
t(\overline{\mu_{r}})$. The result now follows by the induction hypothesis.
\end{proof}

We now fix a $d$-exhaustion or $(A,B)$-exhaustion $\{X_n\}_{n \geq 1}$. In 
order to lighten notation we write $F_n$ for $F_{X_n}$.

\begin{prop}
\label{prop:chasing}
Let $r \geq 0$. Let $(\mu_n)_{n \geq r}$ be a sequence with $\mu_n \in 
\Pieri{\wt\lambda_n}{F_n}$ and $\map{\mu_n}{\mu_{n+1}}{X_n}{X_{n+1}} \neq 0$. 
\begin{enumerate}
  \item If $\mathbf F$ is a $\lie b$-highest weight module then
    the sequence $(\overline{\mu_n})_{n \geq r}$ stabilizes.

  \item If the sequence $\overline{\mu_n}$ does not stabilize then there exist 
    $m > n$ such that $(\mu_n^{X_{n+1}})' \prec \mu_{n+1}$ and 
    $\map{(\mu_n^{X_{n+1}})' }{\mu_m}{X_{n+1}}{X_m} \neq 0$. 
\end{enumerate}
\end{prop}
\begin{proof}
Suppose $\mathbf F$ is a $\lie b$-highest weight module with highest 
weight $\wt\mu$. By Propositions \ref{prop:d-exhaustion} and 
\ref{prop:AB-exhaustion} we can assume that $X_n^- = \emptyset$ whenever 
$(X_n,X_{n+1})$ is a $d$-step or a $B$-step, and that $X^+_n = \emptyset$ 
whenever $(X_n,X_{n+1})$ is an $A$-step. Thus by Propositions 
\ref{prop:d-nonzero} and \ref{prop:AB-nonzero}, $\map{\mu_n}{\mu_{n+1}}{X_n}
{X_{n+1}} \neq 0$ implies 
\[
  h_\infty(\wt\mu,\overline{\mu_{n+1}})
    = h_\infty(\wt\mu,\overline{\mu_n}) - 
      h_\infty(\overline{\mu_{n+1}},\overline{\mu_n})
    \leq h_\infty(\wt\mu,\overline{\mu_n}).
\]
Since $h_{\infty}(\wt\mu,\overline{\mu_n}) \in \ZZ_{>0}$ it must be the case 
that $h_{\infty}(\overline{\mu_{n+1}},\overline{\mu_n}) = 0$ for $n \gg 0$. 
Thus the sequence $\overline{\mu_n}$ must eventually stabilize at $\wt\gamma$. 
This takes care of item $(i)$.

Suppose now that the sequence $\overline{\mu_n}$ does not stabilize, so there 
are infinitely many $n \in \ZZ_{>0}$ such that $\overline{\mu_{n+1}} \succ 
\overline{\mu_n}$. Let $\nu_n$ be the sequence defined by $\nu_r = 
\mu_r$ and $\nu_{n+1} = (\nu_n^{X_{n+1}})'$. By Lemmas 
\ref{lem:d-compatible-pieri} and \ref{lem:AB-compatible-pieri} the sequence
$(\overline{\nu_n})_{n\geq r}$ stabilizes. Since $\overline{\mu_n}$ does not 
stabilize, there must be infinitely $n$ for which $\mu_{n+1} \succ 
(\mu_n^{X_{n+1}})'$.

We consider now the subcase where $\mu_{n+1} \succ \nu_{n+1} = 
(\mu_n^{X_{n+1}})'$ and $h_\infty(\overline{\mu_{n+1}}, \overline{\nu_{n+1}})
> 0$ for some $n$. Using Lemma \ref{lem:approximation} we define recursively 
$\nu_{n+k} \in \Pieri{\wt\lambda_{n+k}}{F_{n+k}}$ such that $\map{\nu_{n+1}}
{\nu_{n+k}}{X_{n+1}}{X_{n+k}} \neq 0$ and $\overline{\mu_{n+k}} - \overline
{\nu_{n+k}} = \epsilon_{I^{n+k}} - \epsilon_{J^{n+k}}$ with $I^{n+k} \subset 
X^-, J^{n+k} \subset X^+ \sqcup Z$, so $\mu_{n+k} \ggcurly \nu_{n_k}$. By 
construction $h_\infty(\overline{\nu_{m+1}}, \overline{\mu_{m+1}}) = 0$, and so 
$\nu_{m+1} = \mu_{m+1}$ and the result follows.

We are left with the case where the sequence $\overline\mu_n$ does not 
stabilize and $h_\infty(\overline{\mu_{n+1}}, \overline{\mu_n}) = 0$
for all $n \geq r$. By the previous lemma we only have to consider the cases 
$\mathbf F = \SS^d\VV$ or $\mathbf F = \bigwedge^d \VV$, and so $\# \supp 
\mu_n = d$. Furthermore, whenever $\overline{\mu_{n+1}} \succ \overline{\mu_n}$ 
we must have $X_{n+1} = Z_n \sqcup X_n$ and $\overline{\mu_{n+1}} - 
\overline{\mu_n} = \epsilon_I - \epsilon_J$ with $I \subset Z_n$ and $J 
\subset X_n$. In particular, $\supp \mu_{n+1} \cap Z_n \neq \emptyset$ and 
$\supp \mu_{n+1} \cap X_n \subsetneq \supp \mu_n$.

Consider a sequence $n_0, n_1, \ldots, n_d$ such that $\mu_{n_{s}} 
\succ (\mu_{n_{s-1}}^{X_{n_s}})'$ and set $m = n_d$ and $\nu = 
\mu_{n_d}$. By the pigeonhole principle either $\supp\wt\nu \cap X_{n_0} = 
\emptyset$, in which case we take $s = 0$, or there exists an $s \in \interval 
d$ such that $\supp\wt\nu \cap Z_s = \emptyset$. Setting $n = n_s$ and 
$\wt\gamma = (\mu_{n_s}^{X_{n_s+1}})'$ we have $\wt\nu - \wt\gamma = 
\epsilon_{P} - \epsilon_{Q}$ for sequences $Q$ and $P$ such that $Q \subset 
X_{n_s}$ and $\min P > \max Z_{s} = \max X_{n_s + 1}$. Thus 
$\map{\gamma}{\nu}{X_n}{X_m} \neq 0$ by item $(ii)$ of Proposition 
\ref{prop:nonzero-sufficient}.
\end{proof}

\subsubsection{The non-example}
\label{ex:non-ex}
We remark that Proposition \ref{prop:chasing} does not apply when $\mathbf F 
= \bigwedge^A \VV$ but $A$ is not $\JJ_\infty$-initial. To see this, let $\II 
= \ZZ_{>0}, \wt\lambda = \sum_{i=1}^\infty -i \epsilon_i$ and $A = 2 \ZZ_{>0}$,
We get an $(A,B)$-exhaustion by taking $X_n = \interval n$ with $(X_n,X_{n+1})$
a $B$-step whenever $n$ is even, and an $A$-step whenever $n$ is odd. 

Set $\mu_n = \epsilon_1 + \epsilon_3 + \cdots + \epsilon_{2k-1}$ with $k = \# 
(A \cap \interval n)$. Then $\mu_n \in \Pieri{\wt\lambda}{F_n}$ and furthermore
\begin{align*}
  \overline \mu_{n+1} - \overline \mu_n 
    &= \begin{cases}
      0 & \mbox { if $n$ is even } \\
      \epsilon_n - \epsilon_{n+1}   & \mbox { if $n$ is odd }.
    \end{cases} 
\end{align*}
In both cases $\map{\mu_n}{\mu_{n+1}}{X_n}{X_{n+1}} \neq 0$ and 
$(\mu_n^{X_{n+1}})' = \mu_n^{X_n}$, while $\mu_{n+1} \succ \mu_n^{X_{n+1}}$ if 
and only $n$ is odd. On the other hand, using Propositions \ref{prop:d-nonzero}
and \ref{prop:AB-nonzero} it is easy to prove that $\map{X_{n+1}}{X_m}
{\mu_{n}^{X_{n+1}}}{\nu} \neq 0$ implies $k+1 \in \supp \nu$, while $k+1 \notin 
\supp \mu_m$ for any $m > n+1$. The consequences of this failure will be seen in
\ref{ex:non-example-bis}.


\section{Proofs of the main theorems}
We are now ready to establish the necessary results on the structure of 
the modules $\mathbf L(\wt\lambda) \otimes \mathbf F$ in order to prove 
theorems A through D.

Fix a weight $\wt\lambda \in \lie h^*$, and denote by $\mathbf F$ 
one of $\SS^d\VV, \bigwedge^d \VV$ for $d \in \ZZ_{>0}$, or $\bigwedge^A \VV$ 
for some infinite set $A$ with $B = \II \setminus A$ infinite. When $\mathbf 
F$ is a Fock module, we assume that $A$ is $\JJ_\infty$-initial. We fix a 
$(\wt\lambda, \mathbf F)$-compatible total order $\preceq$ on $\II$, and set 
$\lie b = \lie b_{\preceq}$ and $\mathbf M = \mathbf L(\wt\lambda) \otimes 
\mathbf F$. 

We fix a $d$-exhaustion or $(A,B)$-exhaustion $\{X_n\}_{n \geq 1}$ of $\II$. To 
lighten up notation we write $\wt\gamma_n$ for $\wt\gamma_{X_n}$ and 
$L_n(\wt\lambda_n)$ for $L_{X_n}(\wt\lambda_{X_n})$. We fix inductive systems 
of $\gl(X_n)$-modules $F_n$ such that $\varinjlim F_n = \mathbf F$ as in 
subsections \ref{ss:tensor-setup} and \ref{ss:fock-setup}. With this setup 
$\mathbf M = \varinjlim L_n(\wt\lambda_n) \otimes F_n$, and we identify $M_n 
:= L_n(\wt\lambda_n) \otimes F_n$ with its image in $\mathbf M$.

\begin{defn}
Let $\gamma \in \Pieri{\wt\lambda}{F_n}$. We denote by $\mathbf M(n, \gamma)$ 
the submodule of $\mathbf M$ generated by $L_n(\wt\lambda_n + \gamma) \subset 
M_n$. Then we observe that 
$\mathbf M(n,\gamma)$ equals the sum of all $L_m(\wt\lambda_m+\nu)$ such that 
$\map{\gamma}{\nu}{X_n}{X_m} \neq 0$. We denote by $\mathbf K(n,\gamma) 
\subset \mathbf M(n,\gamma)$ the sum of all $L_m(\wt\lambda_m + \wt\nu)$ with 
$\map{\gamma}{\nu}{X_n}{X_m} \neq 0$ and $\nu \neq \gamma^{X_m}$.
\end{defn}

It is immediate from the observation above that $\mathbf M(m,\nu) \subset 
\mathbf M(n,\gamma)$ if and only if $\map{\gamma}{\nu}{X_n}{X_m} \neq 0$. If 
$\mathbf 
S \subset \mathbf M$ is a submodule then $\mathbf S = \sum \mathbf 
M(n,\gamma)$, where the sum is over all pairs $(n,\wt\gamma)$ such that 
$L_n(\wt\lambda_n + \gamma) \subset \mathbf S$.  If $\wt\gamma \in 
\Pieri{\wt\lambda}{\mathbf F}$ then $\wt\gamma_{X_n} \in \Pieri{\wt\lambda_n}
{F_n}$ for $n \gg 0$, and we set $\mathbf M(n,\wt\gamma): = \mathbf M(n,
\wt\gamma_{X_n})$ and $\mathbf K(n,\wt\gamma): = \mathbf K(n,\wt\gamma_{X_n})$.

\begin{lem}
\label{lem:MK-quotient}
Let $\gamma \in \Pieri{\wt\lambda_n}{F_n}$. Then $\mathbf K(n,\gamma)$ is 
a proper submodule of $\mathbf M(n,\gamma)$ if and only if $\overline{\gamma} 
\in \Pieri{\wt\lambda}{\mathbf F}$, and in that case $\mathbf M(n,\gamma) / 
\mathbf K(n,\gamma) \cong \mathbf L(\wt\lambda + \overline{\gamma})$.
\end{lem}
\begin{proof}
For every $m > n$ we have
\begin{align*}
  \frac{\mathbf M(n,\wt\gamma) \cap M_m}{\mathbf K(n,\wt\gamma) \cap M_m}
    &= \begin{cases}
      L_m(\wt\lambda + \gamma^{X_m})  
        & \mbox{ if } \gamma^{X_m} \in \Pieri{\wt\lambda_m}{F_m} \\
      \{0\}
        & \mbox{ if } \gamma^{X_m} \notin \Pieri{\wt\lambda_m}{F_m}.
    \end{cases}
\end{align*}
By the definition of $d$-exhaustions and $(A,B)$-exhaustions, 
$\overline{\gamma} \in \Pieri{\wt\lambda}{\mathbf F}$ if and only if 
$\gamma^{X_m} \in \Pieri{\wt\lambda}{F_m}$ for all $m \geq n$. On the other 
hand, since direct limits of vector spaces preserve exactness there is an 
isomorphism
\[
  \frac{\mathbf M(n,\gamma)}{\mathbf K(n,\gamma)} 
    \cong \varinjlim_m \frac{\mathbf M(n,\wt\gamma) \cap M_m}{\mathbf K(n,\wt\gamma) \cap M_m}
\]
and the result follows.
\end{proof}

\begin{lem}
\label{lem:ggcurly-M}
Suppose $\wt\gamma, \wt\nu \in \Pieri{\wt\lambda}{\mathbf F}$ and $\wt\nu 
\ggcurly \wt \gamma$. Then for every $n$ there exists $m > n$ such that 
$\mathbf M(m,\wt\nu) \subset \mathbf M(n,\wt\gamma)$.
\end{lem}
\begin{proof}
We proceed by induction on $h_{\infty}(\wt\nu,\wt\gamma)$. If $h_{\infty}
(\wt\nu,\wt\gamma) = 0$ then $\wt\nu = \wt\gamma$ and we can take $m = n + 1$. 
Now suppose $\wt\nu - \wt\gamma = \epsilon_I - \epsilon_J$ with $I \neq 
\emptyset \neq J$, and let $i$ be the last element of $I$ and $i'$ be the 
last element of $J$. It follows that $h_{\infty}(\epsilon_{i}, 
\epsilon_{i'}) > 0$, hence there exists some infinite class $j \in \JJ$ such 
that $\pi_{\JJ}(i) \prec j \preceq \pi_{\JJ}(i')$. 

Choose $k > n$ such that $\pi_\JJ(Z_k) = \{j\}$, and set $\gamma_1 = 
(\wt\gamma_{X_{k+1}} + \epsilon_i - \epsilon_{i'})'$. By Propositions 
\ref{prop:d-nonzero} and \ref{prop:AB-nonzero} we have $\map{\wt\gamma_{X_k}}
{\gamma_1}{X_k}{X_{k+1}} \neq 0$ and so
\begin{align*}
  \mathbf M(k+1,\gamma_1) 
    \subset \mathbf M(k, \wt\gamma)
    \subset \mathbf M(n,\wt\gamma).
\end{align*}
By construction $h_\infty(\wt\nu,\overline{\gamma_1}) < h_{\infty}(\wt\nu,
\wt\gamma)$, and the induction hypothesis implies there exists $m > k$ 
for which $\mathbf M(m,\wt\nu) \subset \mathbf M(k,\gamma_1) \subset 
\mathbf M(n,\wt\gamma)$.
\end{proof}

\subsection{Semisimple submodules and subquotients}

\begin{lem}[Separation lemma]
\label{lem:mk-separation}
Let $\mathbf T \subsetneq \mathbf S \subset \mathbf M$ be submodules.  
There exists $\wt\gamma \in\Pieri{\wt\lambda}{\mathbf F}$ such that $\mathbf 
M(n, \wt\gamma) \subset \mathbf S$ and $\mathbf T \cap \mathbf M(n,
\wt\gamma) \subset \mathbf K(n, \wt\gamma)$.
\end{lem} 
\begin{proof}
Let $T_n := \mathbf T \cap M_n$ and $S_n := \mathbf S \cap M_n$ for all $n \geq 
0$. Then there exists $r \in \ZZ_{>0}$ such that $T_n \subsetneq S_n$ for all 
$n \geq r$, and in particular there exists $\mu_r \in \Pieri{\wt\lambda_r}{F_r}$ 
such that $L_r(\wt\lambda_r + \mu_r) \subset S_r \setminus T_r$. Thus any 
$v \in L_r(\wt\lambda_r + \mu_r)$ is in $\mathbf S$ but not in $\mathbf T$, and 
hence $\phi(X_r,X_n)(v) \in S_n \setminus T_n$. In particular there is at least 
one $\mu_n \in \Pieri{\wt\lambda_n}{X_n}$ with $L_n(\wt\lambda_n + \mu_n)
\subset S_n \setminus T_n$ and $\map{\mu_r}{\mu_n}{X_r}{X_n} \neq 0$. 

We now define recursively a sequence $(\mu_n)_{n \geq r}$ with $\mu_n \in 
\Pieri{\wt\lambda_n}{F_n}$. The first term of the sequence is $\mu_r$ as above, 
and assuming $\mu_n$ is defined we let $\mu_{n+1}$ be a minimal element in 
the set of all weights in $\Pieri{\wt\lambda_{n+1}}{F_n}$ such that 
$\map{\mu_n}{\mu_{n+1}}{X_n}{X_{n+1}} \neq 0$ and $L_{n+1}(\wt\lambda_{n+1} + 
\mu_{n+1}) \subset S_{n+1} \setminus T_{n+1}$. We claim that there exists 
$\wt\gamma \in \Pieri{\wt\lambda}{\mathbf F}$ with $\mu_m = 
\wt\gamma_{X_m}$ for $m \gg 0$. Indeed, Proposition \ref{prop:chasing} states that if there is no such 
$\wt\gamma$ then there exist $n,m$ such that $\map{(\mu_n^{X_{n+1}})'}{\mu_m}
{X_{n+1}}{X_m} \neq 0$ and $(\mu_n^{X_{n+1}})' \prec \mu_n$. But 
$(\mu_n^{X_{n+1}})' \prec \mu_n$ implies that $L_{n+1}(\wt\lambda_{n+1} + 
(\mu_n^{X_{n+1}})') \subset T_{n+1}$ and hence $L_m(\wt\lambda_m + \mu_m)
\subset T_m$, which contradicts the choice of $\mu_m$.
\end{proof}

As an immediate consequence we see that every simple subquotient of $\mathbf M$
is of the form $\mathbf L(\wt\lambda + \wt\gamma)$ with $
\wt\gamma \in \Pieri{\wt\lambda}{\mathbf F}$.
\begin{prop}
\label{prop:simple-submodules}
Let $\mathbf T \subset \mathbf S \subset \mathbf M$. The quotient $\mathbf S/
\mathbf T$ has a submodule isomorphic to $\mathbf L(\wt\lambda+\wt\gamma)$ 
for some $\wt\gamma \in \Pieri{\wt\lambda}{\mathbf F}$
if and only if $\mathbf M(n,\wt\gamma) \subset \mathbf S$ and $\mathbf M(n,
\wt\gamma) \cap \mathbf T = \mathbf K(n,\wt\gamma)$ for $n \gg 0$. Moreover, 
if $\mathbf 
S / \mathbf T$ is simple then $\wt\gamma$ is unique and $\mathbf S / \mathbf T 
\cong \mathbf L(\wt\lambda+\wt\gamma)$.
\end{prop}
\begin{proof}
For the ``if part'', the hypothesis implies that there is a nonzero map $\mathbf
L(\wt\lambda+\wt\gamma) \cong \mathbf M(n,\wt\gamma)/\mathbf K(n,\wt\gamma) 
\to \mathbf S/\mathbf T$. For the ``only if part'', suppose $s \in \mathbf S$ is a 
vector of weight $\wt\lambda+\wt\gamma$ which is a highest weight vector modulo 
$\mathbf T$. Then $s \in S_n = \mathbf S \cap M_n$ for all $n \gg 0$, and it 
must be a highest weight vector modulo $T_n = \mathbf T \cap M_n$. Thus 
the module generated by $s$ is equal to $L_n(\wt\lambda_n+\wt\gamma_n) \oplus 
T'_n$, with $T'_n \subset T_n$ and $L_n(\wt\lambda_n + \wt\gamma_n) \subset 
S_n / T'_n$ for all $n \gg 0$. Consequently $\mathbf M(n,\wt\lambda) \subset 
\mathbf S$, while $\mathbf M(n,\wt\lambda) \cap \mathbf T \subset \mathbf 
K(n,\wt\gamma)$.

The image of $L_n(\wt\lambda_n + \wt\gamma_n)$ under the map $\phi(X_n,X_{n+1})$ 
is contained in $L_{n+1}(\wt\lambda_{n+1} + \wt\gamma_{n+1}) \oplus T'_{n+1}
\subset M_{n+1}$. It follows that if $\map{\wt\gamma}{\wt\nu}{X_n}{X_{n+1}} 
\neq 0$ and $\wt\nu \neq \wt\gamma$ then the image of $\map{\wt\gamma}{\wt\nu}
{X_n}{X_{n+1}}$ lies in $T'_{n+1} \subset \mathbf T$. Therefore, by definition 
$\mathbf K(n,\wt\gamma) \subset \mathbf T$. It follows now from Lemma 
\ref{lem:mk-separation} that every simple subquotient of $\mathbf M$ is 
isomorphic to $\mathbf L(\wt\lambda + \wt\gamma)$ for some $\wt\gamma \in 
\Pieri{\wt\lambda}{\mathbf F}$.
\end{proof}

We now show that a subquotient of $\mathbf M$ is characterized by its simple 
constituents. 
\begin{defn}
Let $\mathbf Q$ be a subquotient of $\mathbf M$. We set
\[
  \mathcal L(\mathbf Q) := \{\wt\gamma \in \Pieri{\wt\lambda}{\mathbf F}
    \mid [\mathbf Q: \mathbf L(\wt\lambda+\wt\gamma)] \neq 0\}.  
\]
\end{defn}

\begin{lem}
\label{lem:support-sets}
Let $\mathbf S, \mathbf T \subset \mathbf M$ be submodules. 
\begin{enumerate}
  \item If $\wt\gamma \in \mathcal L(\mathbf S)$ and $\wt\nu \ggcurly \wt\gamma$
    then $\wt\nu \in \mathcal L(\mathbf S)$.

  \item $\mathcal L(\mathbf T) \subset \mathcal L(\mathbf S)$ if and only if 
    $\mathbf T \subset \mathbf S$.

  \item If $\mathbf T \subset \mathbf S$ then $\mathcal L(\mathbf 
    S / \mathbf T) = \mathcal L(\mathbf S) \setminus \mathcal L(\mathbf T)$.
\end{enumerate}
\end{lem}
\begin{proof}
Suppose $\wt\gamma \in \mathcal L(\mathbf S)$. By Proposition 
\ref{prop:simple-submodules} we have $\mathbf M(n,\wt\gamma) \subset \mathbf S$ 
for some $n \gg 0$ and by Lemma \ref{lem:ggcurly-M} there exists $m > n$ such
that $\mathbf M(m,\wt\nu) \subset \mathbf M(n,\wt\gamma)$. Hence $\wt\nu \in 
\mathcal L(\mathbf S)$ by Proposition \ref{prop:simple-submodules}.

The ``if'' part of the second item is clear. For the ``only if'' part, suppose 
$\mathcal 
L(\mathbf T) \subset \mathcal L(\mathbf S)$ but $\mathbf S \cap \mathbf T 
\subsetneq \mathbf T$. By the separation lemma \ref{lem:mk-separation} there 
exists $\wt\gamma \in \mathcal L(\mathbf T)$ such that $\mathbf M(n,
\wt\gamma) \subset \mathbf T$ and $\mathbf M(n,\wt\gamma) \cap \mathbf T \cap 
\mathbf S = \mathbf M(n,\wt\gamma) \cap \mathbf S = \mathbf K(n,\wt\gamma)$ for 
all $n \gg 0$. This implies $\textbf M(m,\wt\gamma) \cap \mathbf S = \mathbf 
K(m,\wt\gamma)$ for all $m \geq n$, and so $\wt\gamma \notin \mathcal L(\mathbf 
S)$ by Proposition \ref{prop:simple-submodules}, a contradiction.

We now prove the third item. By the separation lemma $\wt\gamma \in \mathcal 
L(\mathbf S/\mathbf T)$ if and only if $\mathbf M(n,\wt\gamma) \subset \mathbf 
S$ and $\mathbf M(n,\wt\gamma) \cap \mathbf T \subset \mathbf K(n,\wt\gamma)$ 
for $n \gg 0$, and by Proposition \ref{prop:simple-submodules} this is 
equivalent to $\wt \gamma \in \mathcal L(\mathbf S) \setminus \mathcal 
L(\mathbf T)$. 
\end{proof}

The following result gives necessary and sufficient conditions for a quotient 
$\mathbf S / \mathbf T$ to be semisimple.  
\begin{prop}
\label{prop:ss-quotients}
Let $\mathbf T \subset \mathbf S \subset \mathbf M$ be a chain of submodules.
\begin{enumerate}
  \item If $\mathbf S/\mathbf T$ is semisimple and $\wt\nu, \wt\gamma \in 
    \mathcal L(\mathbf S)$ with $\wt\nu \ggcurly \wt\gamma$ and $\wt\nu \neq 
    \wt\gamma$, then $\wt\nu \in \mathcal L(\mathbf T)$.

  \item Suppose $\mathbf F$ is a $\lie b$-highest weight module, and that 
    $\wt\nu \in \mathcal L(\mathbf T)$ for all $\wt\nu, \wt\gamma \in \mathcal 
    L(\mathbf S)$ with $\wt\nu \ggcurly \wt\gamma$ and $\wt\nu \neq \wt\gamma$. Then $\mathbf S/\mathbf T$ is semisimple.
\end{enumerate}
\end{prop}
\begin{proof}
Suppose $\mathbf S/\mathbf T$ is semisimple. Then by Proposition 
\ref{prop:simple-submodules} for every $\wt\gamma \in \mathcal L(\mathbf S/
\mathbf T)$ there exists $n \gg 0$ such that $\mathbf M(n,\wt\gamma) \subset 
\mathbf S$ and $\mathbf K(n, \wt\gamma) \subset \mathbf T $. Lemma 
\ref{lem:ggcurly-M} implies $\mathbf M(m,\wt\nu) \subset \mathbf K(n,
\wt\gamma)$ for some $m > n$, and so $\wt\nu \in \mathcal L(\mathbf T)$. 

Suppose now $\mathbf F$ is a $\lie b$-highest weight module. If $\mathbf F$ is 
$\SS^d\VV$ or $\bigwedge^d \VV$ then by Proposition
\ref{prop:d-exhaustion} every $d$-step $(X_n, X_{n+1})$ satisfies $X_n^- \neq 
\emptyset$. It follows from Proposition \ref{prop:nonzero-sufficient} that
$\map{\gamma}{\nu}{X_n}{X_{n+1}} \neq 0$ implies $\overline \nu \ggcurly 
\overline \gamma$. Then $\mathbf K(n,\wt\gamma) \subset \mathbf T$ whenever 
$\wt\gamma \in \mathcal L(\mathbf S/\mathbf T)$, and $\mathbf S / \mathbf T$ 
has a submodule isomorphic to $\mathbf M(n,\wt\gamma)/ \mathbf K(n,
\wt\gamma) \cong \mathbf L(\wt\lambda + \wt\gamma)$. Therefore,
$\bigoplus_{\wt\gamma \in \mathcal L(\mathbf S/\mathbf T)} \mathbf L(\wt\lambda 
+ \wt\gamma) \subset \mathbf S / \mathbf T$, and by item $(iii)$ or Lemma 
\ref{lem:support-sets} this inclusion is an equality.

If $\mathbf F = \bigwedge^A \VV$ then by Proposition \ref{prop:AB-exhaustion} 
we can assume that each $B$-step $(X_n, X_{n+1})$ satisfies $X_n^- \neq 
\emptyset$, and every $A$-step satisfies $X_n^+ \neq \emptyset$. The reasoning 
is then completely analogous to the previous case, using Propositions 
\ref{prop:nonzero-sufficient} and \ref{prop:nonzero-sufficient-dual}. 
\end{proof}

\begin{lem}
\label{lem:rad-soc}
For every $\ll \in \LL_{\infty}$ we have $F_\ll \mathbf M = \rad F_{\ll+1} \mathbf M$.
Furthermore, if $F_{\ll} \mathbf M/F_{\ll-1}\mathbf M$ is semisimple then it is 
equal to the socle of $\mathbf M / F_{\ll-1}\mathbf M$.
\end{lem}
\begin{proof}
Let $\mathbf L$ be a simple $\gl(\infty)$-module and let $\pi:F_{\ll+1} \mathbf M 
\to \mathbf L$ be a nonzero morphism. Then $\mathbf L \cong \mathbf 
L(\wt\lambda + \wt\gamma)$ for some $\wt\gamma \in \Pieri{\wt\lambda}{\mathbf 
F}$ by Proposition \ref{prop:simple-submodules}. Furthermore $\{\wt\gamma\} = 
\mathcal L(F_{\ll+1} \mathbf M) \setminus \mathcal L(\ker \pi)$, and by item 
$(i)$ of Proposition \ref{prop:ss-quotients} it follows that the $\LL_\infty$ 
class of $\wt\gamma$ is $d+1$. Consequently $F_{\ll} \mathbf M \subset \ker 
\pi$ and, since $\pi$ was arbitrary, that $F_\ll \mathbf M \subset \rad 
F_{\ll+1}\mathbf M$. On the other hand for every $\wt\gamma \in d+1$ we have a 
map $F_{\ll+1} \mathbf M \to \mathbf L(\wt\lambda + \wt\gamma)$, which implies 
that $\mathcal L(\rad F_{\ll+1} \mathbf M) \subset \mathcal L(F_{\ll}\mathbf 
M)$. It follows from item $(ii)$ of Lemma \ref{lem:support-sets} that $\rad 
F_{\ll+1} \mathbf M \subset F_\ll \mathbf M$, and hence these submodules are 
equal.

If $F_{\ll}\mathbf M/F_{\ll-1}\mathbf M$ is semisimple then it is a submodule 
of $\soc \mathbf M / F_{\ll-1}\mathbf M$. On the other hand, if $\mathbf 
L(\wt\lambda + \wt\gamma)$ is a submodule of $\mathbf M / F_{\ll-1}\mathbf M$ 
then every $\wt\nu \ggcurly \wt\gamma$ must be of $\LL_\infty$-class strictly 
less than $d$ thanks to item $(i)$ of Proposition \ref{prop:ss-quotients}. Thus 
the $\LL_\infty$ class of $\wt\gamma$ is $d$ and $\mathbf L(\wt\lambda + 
\wt\gamma) \subset F_{\ll}\mathbf M/F_{\ll-1}\mathbf M$.
\end{proof}

\subsubsection{The non-example revisited}
\label{ex:non-example-bis}
We go back to the non-example \ref{ex:non-ex}.
Taking $\mathbf K = \sum_{k = 1}^\infty \mathbf K(2k+1, \mu_{2k + 1})$ and 
$\mathbf N = \sum_{k = 1}^\infty \mathbf M(2k+1, \mu_{2k + 1})$, the quotient 
$\mathbf N/\mathbf K$ is isomorphic to $\mathbf L(\wt\lambda + \epsilon_B)$. 
Since $\epsilon_B \notin \supp \mathbf F$, both the separation 
Lemma \ref{lem:mk-separation} and Proposition \ref{prop:simple-submodules}
fail.

\subsection{Proofs of theorems A through D}
\label{ss:proofs}
\begin{proof}[Proof of Theorem A]
Proposition \ref{prop:simple-submodules} implies that every simple constituent 
of $\mathbf M$ is of the form $\mathbf L(\wt\lambda + \wt\gamma)$ with 
$\wt\gamma \in \Pieri{\wt\lambda}{\mathbf F}$, and that $\mathbf M$ is 
multiplicity free follows from item $(iii)$ of Lemma \ref{lem:support-sets}. 

By Lemma \ref{lem:Pieri-varia} the poset 
$\Pieri{\wt\lambda}{\mathbf F}$ is locally finite, and by \cite{FM12}*{Theorem 
1.3} its order can be extended to a total order $\leq$ such that $(\Pieri{
\wt\lambda}{\mathbf F}, \leq)$ is isomorphic to a subset of $\ZZ$.
Let $G_{\wt\gamma} \mathbf M$ be the submodule of $\mathbf M$ generated by $\{
v_{\wt\lambda} \otimes e_{\wt\nu} \mid \wt\nu \geq \wt\gamma$. Since 
$v_{\wt\lambda_n} \otimes e_{\wt\gamma_{X_n}} \in G_{\wt\gamma} \mathbf M$
for all $n \gg 0$ it follows that $\mathbf M(n,\wt\gamma) \subset G_{\wt\gamma} 
\mathbf M$, and since $v_{\wt\lambda_n} \otimes e_{\wt\gamma_{X_n}} \notin 
G_{\wt\gamma-1} \mathbf M$ it follows that $\mathcal L(G_{\wt\gamma} \mathbf 
M/G_{\wt\gamma-1} \mathbf M) = \{\wt\gamma\}$, so $\{G_{\wt\gamma} \mathbf M
\}_{\wt\gamma \in \Pieri{\wt\lambda}{\mathbf F}}$ is a composition series
of $\mathbf M$.

If $\wt\nu \ggcurly \wt\gamma$ then $\mathbf L(\wt\lambda + \wt\nu)$ is linked
to $\mathbf L(\wt\lambda + \wt\gamma)$ by Lemma \ref{lem:ggcurly-M}. Suppose 
now that $\mathbf L(\wt\lambda+\wt\nu)$ is linked to $\mathbf L(\wt\lambda+
\wt\gamma)$, and choose $n \gg 0$ such that $\wt\gamma_i = \wt\nu_i$ for all 
$i \in \II \setminus X_n$. Then $[\mathbf M(n, \wt\lambda): \mathbf 
L(\wt\lambda + \wt\nu)] \neq 0$, so by Proposition \ref{prop:simple-submodules} 
there exists $m \geq n$ such that $\mathbf M(m,\wt\nu) \subset \mathbf 
M(n,\wt\gamma)$, and hence $\map{\wt\gamma_{X_n}}{\wt\nu_{X_m}}{n}{m} \neq 0$.
The choice of $n$ plus Corollaries \ref{cor:d-nonzero-maps} and 
\ref{cor:AB-nonzero-maps} imply that $\wt\nu - \wt\gamma = \wt\mu$ with
$\wt\mu \ggcurly 0$, so $\wt\nu \ggcurly \wt\gamma$.
\end{proof}

\begin{proof}[Proof of Corollary A]
It is an immediate corollary of Theorem A that the length of $\mathbf M$ is 
equal to the cardinality of $\Pieri{\wt\lambda}{\mathbf F}$. As we have seen in 
Lemma \ref{lem:Pieri-varia}, if $\mathbf F = \SS^d \VV$ or $\mathbf F = 
\bigwedge^d \VV$, we have an embedding $\Pieri{\wt\lambda}{\mathbf F} \to 
\JJ^d$, and so $\Pieri{\wt\lambda}{\mathbf F}$ is finite whenever $\JJ$ is 
finite. On the other hand, if $\JJ$ is infinite then there are infinitely many 
sequences $(j_1, \ldots, j_d)$ with $j_k \in \JJ$ and $j_k \prec j_{k+1}$ for 
all $k \in \interval{d-1}$. Taking $i_k$ to be the first element of the class 
$j_k$ and $I = (i_1, \ldots, i_d)$ we have $\epsilon_I \in \Pieri{\wt\lambda}
{\mathbf F}$, and hence $\Pieri{\wt\lambda}{\mathbf F}$ is infinite.

Suppose now that $\mathbf F = \bigwedge^A \VV$ and $\mathbf M$ has finite 
length. We claim that $\JJ$ must be a finite set. Let $I \subset A$ be any 
subset such that $I \cap A \cap j$ is a terminal subset of $A \cap j$ for every
$j \in \JJ$. If $\JJ$ is infinite then either $B \cap j$ is infinite for some 
infinite class $j \in \JJ$, in which case we set $J$ to be an initial subset of 
$B \cap j$, or there is an infinite family of classes $\{j_k\}_{k \geq 1} 
\subset \JJ$ such that $B \cap j_k \neq \emptyset$, and we let $J$ be the 
sequence formed by the first elements of $B\cap j_1, B\cap j_2, \ldots, B \cap 
j_{\#I}$. In either case $\epsilon_A - \epsilon_I + \epsilon_J \in 
\Pieri{\wt\lambda}{\mathbf F}$, and since there are infinitely many choices for 
$I$ the set $\Pieri{\wt\lambda}{\mathbf F}$ is infinite, a contradiction. Thus 
if $\mathbf M$ has finite length then $\JJ$ must be finite. The case where $\JJ$
is finite is analyzed in Example \ref{ex:modules_with_finite_jj_}, where we 
showed that $\mathbf M$ has finite length if and only if $\mathbf L(\wt\lambda)
\cong \mathbf D_n \otimes \mathbf V^{\lambda,\mu}$.
\end{proof}

\begin{proof}[Proof of Theorem B]
Suppose $\mathbf F$ is a $\lie b$-highest weight module. Then $\wt\mu \in 
\Pieri{\wt\lambda}{\mathbf F}$ and its $\LL_\infty$-class $d$ is the minimal
element of $\LL_\infty$. It follows that the linkage filtration has a minimal
layer, and by item $(ii)$ of Proposition \ref{prop:ss-quotients} its layers
are semisimple. A simple induction using Lemma \ref{lem:rad-soc} now shows
that the linkage filtration and the socle filtration coincide. In particular,
$\mathbf M$ is semisimple if and only if the linkage filtration has length one, 
i.e., if and only if $\LL_\infty$ is trivial. This proves item $(i)$.

We now prove item $(ii)$. Suppose $\mathbf M = \mathbf S + \mathbf T$, with 
$\mathbf S$ and $\mathbf T$ nontrivial submodules. Again the hypothesis that 
$\mathbf F$ has a highest weight implies that $\mathbf M$ has a 
nontrivial socle, and hence so do $\mathbf S$ and $\mathbf T$. Let $\wt\gamma
\in \mathcal L(\soc \mathbf S)$ and $\wt\nu \in \mathcal L(\soc \mathbf T)$,
and let $I, I'$ be sequences of minimal length such that $\wt\nu - \wt\gamma =
\epsilon_I - \epsilon_{I'}$. We will prove that $\mathbf S \cap \mathbf T \neq 
\{0\}$ by induction on the length $r$ of $I$ and $I'$.

If $r = 0$ then by definition $\mathbf S \cap \mathbf T$ contains the unique 
submodule of $\soc \mathbf M$ isomorphic to $\mathbf L(\wt\lambda + \wt\gamma)$,
and in particular $\mathbf S \cap \mathbf T$ is nonzero. If $r = 1$ then $\wt\nu - 
\wt\gamma = 
\epsilon_i - \epsilon_{i'}$ and without loss of generality we can assume $i 
\prec i'$. Since $\LL_\infty$ has at least two elements the same holds for 
$\JJ_\infty$, and hence there exists an infinite class $j \in \JJ$ such that 
its minimal element $s(j)$ is strictly larger than any element of $\supp 
\wt\gamma \cup \supp \wt\nu$. By construction $\wt\sigma = \wt\gamma - 
\epsilon_{i'} + \epsilon_{s(j)} = \wt\nu - \epsilon_i + \epsilon_{s(j)}$ 
belongs to $\Pieri{\wt\lambda}{\mathbf F}$, and $\wt\gamma, \wt\nu \ggcurly 
\wt\sigma$. As $\Pieri{\wt\lambda}{\mathbf F} = \mathcal L(\mathbf S) \cup 
\mathcal L(\mathbf T)$, we can assume without loss of generality $\wt\sigma
\in \mathcal L(\mathbf S)$. Item $(i)$ of Proposition \ref{lem:support-sets} 
then implies that $\wt\gamma, \wt\nu \in \mathcal L(\mathbf S)$ and so $\mathbf 
L(\wt\lambda + \wt\nu) \subset \mathbf S \cap \mathbf T$.

To proceed with the induction, we assume $r > 1$ and let $i = \max I, i' = \max 
I'$. Set $\wt\sigma = \wt\nu - \epsilon_{i'} + \epsilon_i$. Again $\wt\sigma 
\in \mathcal L(\mathbf S) \cup \mathcal L(\mathbf T)$. If $\wt\sigma \in 
\mathcal L(\mathbf S)$ we replace $\wt\gamma$ by $\wt\sigma$ and conclude 
from the case $r = 1$ that $\mathbf S \cap \mathbf T \neq \{0\}$. If $\wt\sigma 
\in \mathcal L(\mathbf T)$ we replace $\wt\nu$ by $\wt\sigma$ and the 
result now follows from the induction hypothesis. This completes the proof of 
item $(ii)$.

To prove item $(iii)$, let $\mathbf S, \mathbf T \subset \mathbf M$ be arbitrary
nontrivial submodules and let $\wt\gamma \in \mathcal L(\mathbf S), \wt\nu \in 
\mathcal L(\mathbf T)$. Again take sequences $I,I'$ of minimal length $r$ such 
that $\wt\nu - \wt\gamma = \epsilon_I - \epsilon_{I'}$.
For $n \gg 0$ we have $I, I' \subset X_n$, and
\begin{align*}
  L_n(\wt\lambda_{X_n} + \wt\gamma_{X_n}) &\subset \mathbf S,
    &   L_n(\wt\lambda_{X_n} + \wt\nu_{X_n}) &\subset \mathbf T.
\end{align*}
The hypothesis that $\mathbf F$ is not a $\lie b$-highest weight module implies 
that there are infinitely many $n$ such that $X_{n+1} = Z_n \sqcup X_n$ with 
$Z_n \cap \supp \wt\gamma = Z_n \cap \supp \wt\nu = \emptyset$. We fix such an 
$n$ and let $i = \max I, i' = \max I'$ and $z = \min Z_n$. Item $(ii)$ of 
Proposition \ref{prop:d-nonzero} implies 
\begin{align*}
  L_{n+1}(\wt\lambda_{X_{n+1}} + \wt\gamma_{X_{n+1}} - \epsilon_i + \epsilon_z) 
    &\subset \mathbf S,
    &   L_{n+1}(\wt\lambda_{X_{n+1}} + \wt\nu_{X_{n+1}} - \epsilon_{i'} + 
      \epsilon_z) 
    &\subset \mathbf T,
\end{align*}
so $\wt\gamma_1 = \wt\gamma - \epsilon_i + \epsilon_z \in \mathcal L(\mathbf 
S)$ and $\wt\nu_1 = \wt\nu - \epsilon_i + \epsilon_z \in \mathcal L(\mathbf T)$.
By construction $\wt\nu_1 - \wt\gamma_1 = \epsilon_{I_1} - \epsilon_{I'_1}$ 
with $I_1, I'_1$ of length $r-1$. Proceeding recursively we can find 
$\wt\gamma_k \in \mathcal L(\mathbf S)$ and $\wt\nu_k \in \mathcal L(\mathbf 
T)$ such that $\wt\nu_k - \wt\gamma_k = \epsilon_{I_k} - \epsilon_{I'_k}$ with 
the length of $I_k, I'_k$ equal to $r - k$. After $r$ steps we reach 
$\wt\gamma_r = \wt\nu_r \in \mathcal L(\mathbf S) \cap \mathcal L(\mathbf T) = 
\mathcal L(\mathbf S \cap \mathbf T)$. Thus $\mathbf S \cap \mathbf T$ is 
nonzero. 

In the course of the proof we showed that for every submodule $\mathbf S \subset
\mathbf M$ with $L_n(\wt\lambda_{X_n} + \wt\gamma_{X_n}) \subset \mathbf S$ we 
also have $L_{n+1}(\wt\lambda_{X_{n+1}} + \wt\gamma_{X_{n+1}} - \epsilon_i + 
\epsilon_z) \subset \mathbf S$. This implies that $\mathbf K(n+1,
\wt\gamma_{X_{n+1}} - \epsilon_i + \epsilon_z)$ is a strict submodule of 
$\mathbf S$. Thus every submodule of $\mathbf M$ has a nontrivial submodule, 
and hence $\mathbf M$ has no socle.
\end{proof}

\begin{proof}[Proof of Theorem C]
If $\mathbf F$ is a $\lie b$-highest weight module then the socle filtration 
and the linkage filtration of $\mathbf M$ coincide by Theorem B, and since the 
linkage filtration is exhaustive so is the socle filtration. On the other hand, 
if $\mathbf F$ is not a $\lie b$-highest weight module then $\soc \mathbf M = 
\{0\}$ by item $(iii)$ of Theorem B. This proves item $(i)$. Notice in 
particular that $\LL_\infty$ has a minimum if and only if $\mathbf F$ is a 
$\lie b$-highest weight module.

If $\LL_\infty$ has a maximum $\ll$, Lemma \ref{lem:rad-soc} implies $F_{\ll-1} 
\mathbf M = \rad \mathbf M$. On the other hand, if $\rad \mathbf M \subset 
\mathbf M$ then the $\LL_\infty$ class of any $\wt\gamma \in \mathcal L(\mathbf 
M/\rad \mathbf M)$ is the maximum of $\LL_\infty$. Thus $\rad \mathbf M$ is a 
strict submodule of $\mathbf M$ if and only if $\LL_\infty$ has a maximum. 
Lemma \ref{lem:rad-soc} plus descending induction show that the linkage 
filtration and the radical filtration coincide, and since the linkage 
filtration is separated, so is the radical filtration. This proves item $(ii)$.
Item $(iii)$ follows from $(i)$ and $(ii)$. 
\end{proof}

\begin{proof}[Proof of Theorem D]
The proof follows immediately from the fact that $-_*$ is an autoequivalence
of the tensor category of integrable $\gl(\infty)$-modules. In particular note 
that $\mathbf L(\wt\lambda) \otimes \mathbf F_* \cong (\mathbf L(-\wt\lambda) 
\otimes \mathbf F)_*$, and that $\mathbf F_*$ is a highest weight module with 
respect to a Borel subalgebra arising from a $(\wt\lambda,\mathbf 
F_*)$-compatible order if and only if $\mathbf F$ is a highest weight module 
with respect to a Borel subalgebra arising from a $(-\wt\lambda,\mathbf 
F)$-compatible order.
\end{proof}

\end{document}